\documentclass[english]{amsart}

\usepackage{mathtools}
\usepackage{accents}
\usepackage{blkarray}
\usepackage{multirow}
\usepackage{bbm}  
  
\newtheorem{thm}{Theorem}[section]

\newtheorem{defn}[thm]{Definition}
\newtheorem{rem}[thm]{Remark}
\newtheorem{fact}[thm]{Fact}

\newtheorem{assumption}[thm]{Assumption}
 \newtheorem{prop}[thm]{Proposition}
\newtheorem{lem}[thm]{Lemma}
\newtheorem{cor}[thm]{Corollary}

\usepackage[noadjust]{cite}

\makeatletter

 \thanks{
	This research was supported  by the German Research Foundation (DFG) grant {\em Renewal Theory and Statistics of Rare Events in Infinite Ergodic Theory} (Gesch\"aftszeichen KE 1440/2-1).
}

\begin{document}

\title[Escape rates for special flows]{Escape rates for special flows and their higher order asymptotics}

\author{Fabian Dreher and Marc Kesseb\"ohmer	 }

 \date{\today}

\address{
	Fachbereich~3 -- Mathematik und Informatik, Universit\"at Bremen, Bibliothekstra{\ss}e~1, 28359~Bremen, Germany
}
 \email{fdreher@uni-bremen.de}\email{mhk@uni-bremen.de}

 
\maketitle

\begin{abstract}
	In this paper  escape rates and local escape rates for special flows are sudied. In a general context the first result is that the escape rate depends monotonically on the ceiling function and fulfils certain scaling, invariance, and continuity  properties.  For the metric setting local escape rates are considered. If the base transformation is ergodic and exhibits an exponential convergence in probability of ergodic sums, then the local escape rate with respect to the flow is just the local escape rate with respect to the base transformation, divided by the integral of the ceiling function. Also a reformulation with respect to induced pressure is presented. Finally, under additional regularity conditions higher order asymptotics for the local escape rate are established.	
\end{abstract}

\global\long\def\ind{\mathbbm{1}}
\global\long\def\complement{\mathsf{c}}

\global\long\def\dx#1{\,\mathrm{d}#1}
\global\long\def\L{\mathcal{L}}
\global\long\def\vect#1{\accentset{\rightharpoonup}{#1}}
\global\long\def\vectfill#1{\overrightharpoon{#1}}
\global\long\def\Log{\operatorname{Log}}
\global\long\def\diag{\operatorname{diag}}
\global\long\def\oo#1{\operatorname{o}\bigl(#1\bigr)}
\global\long\def\OO#1{\operatorname{O}\bigl(#1\bigr)}
\global\long\def\bbar#1{{\mkern1.5mu\overline{\mkern-1.5mu#1\mkern-1.5mu}\mkern1.5mu}}
\global\long\def\negspace#1#2{\hspace{-#1}#2\hspace{#1}}
\global\long\def\var{\operatorname{var}}
\global\long\def\1{\mathbbm{1}}

\section{Introduction}

A dynamical system is called open if over time mass is leaking from it. A closed system can be turned into an open one by designating a subset $A$ of the phase space $X$ as a hole. The asymptotic speed with which the system is leaking mass is measured by the escape rate $\rho\left(A\right)$. In addition to the study of conditionally invariant measures \cite{Pianigiani:1979,Collet:1997,Liverani:2003}, there has been an increased interest in the dependency of the escape rate on the size and position of the hole \cite{Keller:2009,Bunimovich:2011,Ferguson:2012,Bandtlow:2014}.

This paper investigates escape rates in the context of flows. It is natural to consider ergodic systems, because ergodicity implies that eventually the entire (finite) mass escapes through any hole of positive measure -- a mandatory condition for a non-zero escape rate. A fundamental result by Ambrose \cite{Ambrose:1941} states that every measurable ergodic flow is isomorphic to a flow built under a function, also referred to as special flow or suspension flow. It is a consequence of that result that this paper emphasises the study of escape rates for special flows with a base transformation $\theta:X\rightarrow X$ and a ceiling function $\varphi:X\rightarrow\mathbb{R}$ (Definition~\ref{def:Stroemung-unter-einer-Funktion}) and focuses in particular on the question of how the escape rate changes when the ceiling function is altered while the base transformation remains fixed.

Results are obtained in several different settings that impose successively increasing restrictions on base transformation, ceiling function and shape of the hole.

The first setting only requires basic measurability and integrability conditions on base transformation and ceiling functions. In this situation, elementary results about the relation between the escape rate and the hole are proved (Proposition~\ref{prop:EigenschaftenAR_Mengen}) -- the most notable being the observation that the escape rate of a hole in a special flow with bounded ceiling depends only on the shadow of the hole in the base. Furthermore it is shown (Proposition~\ref{prop:EigenschaftenAR_Dach}) that the escape rate depends monotonically on the ceiling function, fulfils a scaling property and is invariant under addition of coboundaries. Also, the escape rate is continuous with respect to the supremum norm (Proposition~\ref{prop:AR_glmKonvergenz}).

In Section~\ref{chap:LokaleAR}, the base is required to be a metric space. This allows one to define the local escape rate which relates the escape rate to the measure of the hole for the case of holes shrinking towards a point. If the base transformation is ergodic and exhibits an exponential convergence of ergodic sums, then the local escape rate with respect to the flow is just the local escape rate with respect to the base transformation, divided by the integral of the ceiling function (Theorem~\ref{thm:LokaleAR_IntegralMalBasis}).  In particular, this holds if the base transformation is  a Markov shift or more generally a weak invariant Gibbs measure as introduced in Section~\ref{sec:weakThermo}. 

In Section~\ref{chap:AR-als-induzierter-Druck}, it is shown that the escape rate of a hole that can be written as a union of cylinder sets of bounded length can be rephrased as an induced pressure in the sense of Jaerisch, Kesseb\"ohmer and Lamei \cite{Jaerisch:2014} if the base transformation is a symbolic shift on an at most countable alphabet with  weak Gibbs probability measure (Proposition~\ref{prop:AR-als-induzierter-Druck}). Properties of the induced pressure then allow one to deduce the sublinearity of the reciprocal escape rate in that situation (Corollary~\ref{cor:AR-sublinear}).

Finally, the most restrictive setting is used in Section~\ref{chap:Das-Umkehrproblem} where the base transformation is required to be a Markov shift on a finite alphabet and the ceiling function is required to be constant on cylinder sets of a certain length. In this section, the results of Cristadoro, Knight and Degli Esposti \cite{Cristadoro:2013} are expanded upon in order to prove higher order asymptotics for the escape rate of shrinking holes (Theorems~\ref{thm:Cristadoro_Nullstelle_kleinO}~and~\ref{thm:Explizite-erste-zwei-Ordnungen}). This also allows one to recover information about the orbit lengths of periodic points purely by considering escape rates and the measure of shrinking holes.

The paper is based on the first author's dissertation \cite{Dreher:2015} that was supervised by the second author.

\section{Preliminaries}\label{chap:Vorbereitungen}

Whenever the real numbers $\mathbb{R}$ are used, they are understood to be equipped with the Borel $\sigma$-algebra and the Lebesgue measure.  
The natural numbers $\mathbb{N}$ are defined as the set $\left\{ 1,2,3,\dots\right\} $, and $\mathbb{N}_{0}$ refers to the set $\mathbb{N}\cup\left\{ 0\right\}$. The indicator function of a set $A\subset X$ is written as $\ind_{A}$, and the indicator function of the whole space $X$ is abbreviated as $\ind$.

\subsection{Special flows}

Consider a map $\theta:X\rightarrow X$, a function $\varphi:X\rightarrow\mathbb{R}$ and an element $x\in X$. The following notations will be used:
\begin{align}
n\in\mathbb{N}:\quad S_{n}\varphi\left(x\right) & \coloneqq\sum_{k=0}^{n-1}\varphi\circ\theta^{k}\left(x\right)\label{eq:Def_Sn_Nt_NA}\\
t\in\mathbb{R}_{\geq0}:\quad N_{t}^{\varphi}\left(x\right) & \coloneqq\min\left\{ n\in\mathbb{N}_{0}\middle|\; S_{n}\varphi\left(x\right)>t\right\} \nonumber \\
A\subset X:\quad N_{A}\left(x\right) & \coloneqq\min\left\{ n\in\mathbb{N}_{0}\middle|\; \theta^{n}x\in A\right\} \nonumber 
\end{align}
If the function $\varphi$ can be inferred from the context, this can be shortened to $N_{t}\left(x\right)$, and in case of terms like $S_{N_{t}\left(x\right)}\varphi\left(x\right)$ even further to $S_{N_{t}}\varphi\left(x\right)$ or $S_{N_{t}}\left(x\right)$.

A \emph{semiflow} is a family $\Phi\coloneqq\left( \Phi_{t}\middle|\; t\in\mathbb{R}_{\geq0}\right)$ of measure-preserving endomorphisms of a measure space $\left(X,\mathcal{A},\mu\right)$ that satisfy
\[
\forall s,t\in\mathbb{R}_{\geq0}:\quad\Phi_{t}\circ\Phi_{s}=\Phi_{s+t}.
\]

\begin{defn}\label{def:Stroemung-unter-einer-Funktion}
Let $\left(X,\mathcal{A},\mu\right)$ be a $\sigma$-finite measure space and let $\theta$ be a measure-preserving endomorphism of $\left(X,\mathcal{A},\mu\right)$. Furthermore, let $\varphi:X\rightarrow\mathbb{R}$ be a measurable positive function that satisfies $\inf_{x\in X}\varphi\left(x\right)>0$. One defines a new measure space $\left(\bbar X,\bbar{\mathcal{A}},\bbar{\mu}\right)$ as a subspace of the product $\left(X,\mathcal{A},\mu\right)\times\left(\mathbb{R},\mathcal{B},\lambda\right)$ via
\[
\bbar X\coloneqq\left\{ \left(x,s\right)\in X\times\mathbb{R}\middle|\; 0\leq s<\varphi\left(x\right)\right\} 
\]
and $\bbar{\mathcal{A}}$, $\bbar{\mu}$ as the corresponding restrictions of the product $\sigma$-algebra and product measure.
On $\left(\bbar X,\bbar{\mathcal{A}},\bbar{\mu}\right)$ one defines a semiflow $\Phi$ as follows:
\[
\Phi_{t}\left(x,s\right)\coloneqq\begin{cases}
\left(x,s+t\right) & ,0\leq t<\varphi\left(x\right)-s\\
\left(\theta^{N_{s+t}\varphi\left(x\right)-1}\left(x\right),s+t-S_{N_{t+s}-1}\varphi\left(x\right)\right) & ,\varphi\left(x\right)-s\leq t.
\end{cases}
\]
\end{defn}

The semiflow $\Phi$ will be referred to as the \emph{flow} (also: special flow or suspension flow) with \emph{ceiling function} $\varphi$ and \emph{base transformation} $\theta$. In this paper, the term ceiling function is taken to imply that the function is measurable, positive and bounded away from zero.
 
\begin{fact}
{\em (\cite[p. 89f]{Jacobs:1960})} The $\Phi_{t}$ are measure-preserving. 
\end{fact}

The projections from  $\bbar X$ onto the first and second component are denoted by $\pi_{1}$ and $\pi_{2}$.

\subsection {Shift spaces}
Shift spaces are an important type of base transformation because they can be used to approximate other transformations and can serve as a useful model to derive hypotheses for more general situations.

Let $S$ be a finite or countable set, equipped with the discrete topology. The space $S^{\mathbb{N}}$, equipped with the product topology and the shift map
\[
\theta:S^{\mathbb{N}}\rightarrow S^{\mathbb{N}},\quad\left(x_{i}\right)_{i\in\mathbb{N}}\mapsto\left(x_{i+1}\right)_{i\in\mathbb{N}},
\]
is a topological dynamical system, called a \emph{(onesided) shift}. The set $S$ is called the \emph{alphabet} of the shift.

The topology on $S^{\mathbb{N}}$ is generated by the \emph{cylinder sets}
\[
\left[a_{1},\dots,a_{n}\right]\coloneqq\left\{ \left(x_{i}\right)_{i\in\mathbb{N}}\in S^{\mathbb{N}}\middle|\; \forall1\leq i\leq n:x_{i}=a_{i}\right\} 
\]
 for $n\in\mathbb{N}$ and
 coincides with the topology generated by the metric
\[
d:S^{\mathbb{N}}\times S^{\mathbb{N}}\rightarrow\mathbb{R},\quad\left(x,y\right)\mapsto e^{-\beta\left|x\wedge y\right|}
\]
for an arbitrary $\beta>0$ and 
\[
\left|x\wedge y\right|\coloneqq\max\left\{ n\in\mathbb{N}_{0}\middle|\; \forall1\leq i\leq n:\, x_{i}=y_{i}\right\} .
\]
Cylinder sets of the form $\left[a_{1},\dots,a_{n}\right]$ are called \emph{$n$-cylinders} or cylinders of length $n$. By convention let the cylinder of length $0$ be $X$.

A transition matrix $P\in\mathbb{R} ^{S\times S}$ defines a \emph{subshift} of $\left(S^{\mathbb{N}},\theta\right)$ by restricting $\theta$ to
\[
\Sigma_{P}\coloneqq S_{P}^{\mathbb{N}}\coloneqq\left\{ \left(x_{i}\right)_{i\in\mathbb{N}}\middle|\; \forall i\in\mathbb{N}:p_{x_{i},x_{i+1}}\neq0\right\} \subset S^{\mathbb{N}}.
\]
 The subshift is called \emph{irreducible} if the matrix $P$ is irreducible.

The set of \emph{admissible} words of length $n$ is denoted by
\[
\Sigma_{P}^{n}\coloneqq\left\{ \left(x_{1},\dots,x_{n}\right)\in S^{n}\middle|\; \forall i\in\left\{ 1,\dots,n\right\} :p_{x_{i},x_{i+1}}\neq0\right\} \subset S^{n}
\]
 and 
\[
\Sigma_{P}^{*}\coloneqq\bigcup_{n\in\mathbb{N}}\Sigma_{P}^{n}
\]
denotes the set of admissible words of arbitrary length. The set of $n$-cylinders of a subshift is denoted by $C_{n}$ and contains all $n$-cylinders that come from admissible words.

\begin{fact}
	If the alphabet of the subshift is finite, the space $S_{M}^{\mathbb{N}}$ is compact with respect to the subspace topology induced by $d$.
\end{fact}

If the transition matrix $P$ defining a subshift on a finite alphabet $S$ is non-negative, irreducible and row-stochastic, it defines a \emph{Markov shift}. In addition to all the properties of a subshift, the system is also equipped with an invariant probability measure $\mu$ that can be derived from the unique positive probability vector $\pi\in\mathbb{R}^{S}$ with $\pi P=\pi$ whose existence can be concluded from the Perron-Frobenius theorem.   The measure $\mu$, which is a Borel measure on $\Sigma_{P}$, is defined by 
\[
\mu\left(\left[a_{1},\dots,a_{n}\right]\right)\coloneqq\pi_{a_{1}}\cdot\prod_{i=1}^{n-1}p_{a_{i},a_{i+1}}
\]
and is invariant with respect to $\theta$.
Note that the irreducibility of the transition matrix $P$ implies that a Markov shift as defined here is automatically ergodic.

A natural choice for well-behaved ceiling functions over Markov shifts are those that are constant on $n$-cylinders:
\[
Z_{n}\coloneqq\left\{ \varphi:X\rightarrow\mathbb{R}\middle|\; \forall C\in C_{n}:\varphi\textrm{ is constant on }C\right\}.
\]
A function $\varphi:X\rightarrow\mathbb{R}$ is called a \emph{cylinder function (of order $n\in \mathbb{N}$)} if  $\varphi\in Z_{n}$. Furthermore, a cylinder function $\varphi$ is called arithmetic if there exists a $\lambda\in\mathbb{R}$ such that $\varphi$ only takes values in the lattice $\lambda\mathbb{Z}$; the function is then also referred to as $\lambda$-arithmetic.

Considering a semiflow over a Markov shift under a $\lambda$-arithmetic ceiling function $\varphi=\lambda\cdot\sum_{C\in C_{n}}k_{C}\cdot\ind_{C}$ with $k_{C}\in\mathbb{N}$, one can partition the space $\bbar X$ into the sets $C\times\left[\left(k-1\right)\lambda,k\lambda\right)$ with $0<k\leq k_{C}$ and $C\in C_{n}$. For these sets that are closely related to cylinder sets, the notation
\[
\bbar C_{n}\coloneqq\left\{ C\times\left[\left(k-1\right)\lambda,k\lambda\right)\middle|\; 0<k\leq k_{C},\, C\in C_{n}\right\} 
\]
shall be used. Analogously to cylinder functions, one defines
\[
\bbar Z_{n}\coloneqq\left\{ f:\bbar X\rightarrow\mathbb{R}\middle|\; \forall\bbar C\in\bbar C_{n}:f\textrm{ ist konstant auf }\bbar C\right\}.
\]

\section{The escape rate and its basic properties}\label{chap:AR_Grundlagen}

The most general setting for escape rates for discrete step dynamical systems is given by a measure-preserving endomorphism  $\theta$ of a probability space $\left(X,\mu\right)$. A \emph{hole} is a measurable set $A\subset X$ for which 
\[
\bigcup_{n\geq0}\theta^{-n}\left(A\right)=X
\]
holds almost surely. If the system $\left(X,\mu,\theta\right)$ -- which will be referred to as the \emph{closed system} -- is ergodic, then every set of positive measure fulfils this condition \cite[Theorem 1.5 (iii)]{Walters:2000}. The system together with the hole $A$ will be referred to as the \emph{open system}.

\begin{defn}
	\label{def:AR_diskret}Let $A\subset X$ be a hole. The \emph{upper (lower) escape rate}
	$\overline{\rho}$ ($\underline{\rho}$) through the hole $A$ is defined as
	\begin{align*}
	\overline{\rho}\left(A\right) & \coloneqq\limsup_{n\rightarrow\infty}-\frac{1}{n}\log\mu\left(\left\{ x\in X\middle|\; \forall0\leq k\leq n:\theta^{k}\left(x\right)\notin A\right\} \right),\\
	\underline{\rho}\left(A\right) & \coloneqq\liminf_{n\rightarrow\infty}-\frac{1}{n}\log\mu\left(\left\{ x\in X\middle|\; \forall0\leq k\leq n:\theta^{k}\left(x\right)\notin A\right\} \right).
	\end{align*}
	If both values coincide, one refers to $\rho\left(A\right)\coloneqq\overline{\rho}\left(A\right)=\underline{\rho}\left(A\right)$ as the \emph{escape rate} of the hole $A$.
\end{defn}

\begin{rem}
	\label{rem:AR-Alternative-Mengenfolge}\label{rem:ARdiskret_reellerParameter}The set $\bigl\{ x\in X|\forall0\leq k\leq n:\theta^{k}\left(x\right)\notin A\bigr\}$ can be replaced by $\bigl\{ x\in X|\forall0\leq k<n:\theta^{k}\left(x\right)\notin A\bigr\}=\left\{ x\in X\middle|\; N_{A}\left(x\right)\geq n\right\} $ without altering $\overline{\rho}$ or $\underline{\rho}$. Also one can use a real parameter $t$ instead of natural numbers $n$.
\end{rem}

The escape rate for a (semi\nobreakdash-)flow $\Phi$ under a function $\varphi$ with base transformation $\theta:X\rightarrow X$ can be defined analogously to the discrete case. Here a \emph{hole} is a measurable set $\bbar A\subset\bbar X$ for which
\[
\bigcup_{t\geq0}\Phi_{t}^{-1}\left(\bbar A\right) = \bbar X
\]
 holds almost surely and 
\[
\bigcup_{t\in\left[0,\tau\right]}\Phi_{t}^{-1}\left(\bbar A\right)
\]
 is measurable for all $\tau\in\mathbb{R}_{\geq0}$.

 The terms closed and open system are used analogously to the discrete case.
\begin{defn}
	\label{def:AR_kontinuierlich}Let $\bbar A\subset\bbar X$ be a hole. The \emph{upper (lower) escape rate}
	$\overline{\rho}$ ($\underline{\rho}$) through the hole $\bbar A$ is defined as
	\begin{align*}
	\overline{\rho}\left(\bbar A,\varphi\right) & \coloneqq\limsup_{t\rightarrow\infty}-\frac{1}{t}\log\bbar{\mu}\left(\left\{ \left(x,s\right)\in\bbar X\middle|\; \forall\tau\in\left[0,t\right]:\Phi_{\tau}\left(x,s\right)\notin\bbar A\right\} \right),\\
	\underline{\rho}\left(\bbar A,\varphi\right) & \coloneqq\liminf_{t\rightarrow\infty}-\frac{1}{t}\log\bbar{\mu}\left(\left\{ \left(x,s\right)\in\bbar X\middle|\; \forall\tau\in\left[0,t\right]:\Phi_{\tau}\left(x,s\right)\notin\bbar A\right\} \right).
	\end{align*}
	If both values coincide, one refers to $\rho\left(\bbar A,\varphi\right)\coloneqq\overline{\rho}\left(\bbar A,\varphi\right)=\underline{\rho}\left(\bbar A,\varphi\right)$ as the \emph{escape rate} of the hole $\bbar A$.
\end{defn}
\begin{rem}
	Rescaling the measure $\bbar{\mu}$ with a factor $\lambda\in\mathbb{R}_{>0}$ does not change the escape rate.
\end{rem}

The basic properties of escape rates for discrete systems (see for example \cite[Proposition 2.3.2]{Bunimovich:2011}) also hold for (semi\nobreakdash-)flows. Additionally, if one considers bounded ceiling functions, the escape rate depends only on the shadow of the hole in the base.
\begin{prop}\label{prop:EigenschaftenAR_Mengen}
	Let $\bbar A,\bbar B\subset\bbar X$ be holes and let $\varphi:X\rightarrow\mathbb{R}$ be a ceiling function. Then the following properties hold:
	\begin{enumerate}
	\item \label{enu:EigenschaftenAR_Mengen_Monotonie}If $\bbar A\subset\bbar B$,
	then $\rho\left(\bbar A,\varphi\right)\leq\rho\left(\bbar B,\varphi\right)$.
	\item \label{enu:EigenschaftenAR_Mengen_Urbild}For all $r\in\mathbb{R}_{\geq0}$
	one has $\rho\left(\bbar A,\varphi\right)=\rho\left(\Phi_{r}^{-1}\left(\bbar A\right),\varphi\right)$.
	\item \label{enu:EigenschaftenAR_Mengen_Urbildvereinigung}For all $r\in\mathbb{R}_{\geq0}$
	one has $\rho\left(\bbar A,\varphi\right) = \rho\bigl(\bigcup_{\tau\in\left[0,r\right]}\Phi_{\tau}^{-1}\left(\bbar A\right),\varphi\bigr)$.
	\item \label{enu:EigenschaftenAR_Mengen_Projektion}If $\varphi$ is bounded and $\pi_{1}\left(\bbar A\right)$ is measurable, then 
	\[\rho\left(\pi_{1}^{-1}\left(\pi_{1}\left(\bbar A\right)\right),\varphi\right)=\rho\left(\bbar A,\varphi\right).\]
	\end{enumerate}
\end{prop}

\begin{proof}
	Properties \eqref{enu:EigenschaftenAR_Mengen_Monotonie}, \eqref{enu:EigenschaftenAR_Mengen_Urbild} and \eqref{enu:EigenschaftenAR_Mengen_Urbildvereinigung} follow from elementary calculations. For property \eqref{enu:EigenschaftenAR_Mengen_Projektion}, one obtains 
	\[
	\rho\left(\pi_{1}^{-1}\left(\pi_{1}\left(\bbar A\right)\right),\varphi\right)\geq\rho\left(\bbar A,\varphi\right)
	\]
	from \eqref{enu:EigenschaftenAR_Mengen_Monotonie}. Due to $\varphi$ being bounded one obtains 
	\[
	\Phi_{\sup\varphi}^{-1}\left(\pi_{1}^{-1}\left(\pi_{1}\left(\bbar A\right)\right)\right)\subset\bigcup_{\tau\in\left[0,2\cdot\sup\varphi\right]}\Phi_{\tau}^{-1}\left(\bbar A\right).
	\]
	From this the opposite inequality for the escape rate follows via \eqref{enu:EigenschaftenAR_Mengen_Monotonie}, \eqref{enu:EigenschaftenAR_Mengen_Urbild} and \eqref{enu:EigenschaftenAR_Mengen_Urbildvereinigung}.
\end{proof}
\begin{rem}
	\label{rem:Loch-Fluss-Basis_reicht}Since in the case of bounded ceiling functions the escape rate is already determined by the shadow of the hole, one simplifies the notation by letting $\rho\left(A,\varphi\right)$ denote  $\rho\bigl(\pi_{1}^{-1}\left(A\right),\varphi\bigr)$ and referring to $A$ as the hole.
\end{rem}

\begin{rem}
	\label{rem:VerschiedeneARtrotzBeschraenktheit}Let $A\subset X$ be a hole such that the upper and lower escape rates for the hole $A$ with respect to the base transformation coincide and are positive and finite. Boundedness of the ceiling function is not sufficient to ensure that upper and lower escape rates for the corresponding special flow with hole $A$ coincide.
	
	One can construct a ceiling function $\varphi$ that exhibits such behaviour by choosing a sequence of positive real numbers $\left(r_{n}\right)_{n\in\mathbb{N}}\in\mathbb{R}_{>0}^{\mathbb{N}}$ which is bounded and bounded away from zero and for which the sequence $\left(\sum_{n=1}^{m}r_{n}/m\right)_{m\in\mathbb{N}}$ has more than one accumulation point, and letting
	\[
	\varphi\coloneqq\ind_{A\cup\left(\bigcup_{n\geq0}\theta^{-n}\left(A\right)\right)^{\complement}}+\sum_{n\in\mathbb{N}}r_{n}\cdot\ind_{\left\{ N_{A}=n\right\} }.
	\]
\end{rem}

Remark~\ref{rem:VerschiedeneARtrotzBeschraenktheit} illustrates that in order to determine the existence of the escape rate in general it does not suffice to consider an arbitrary subsequence of the limit process. It is sufficient though to consider an evenly spaced subsequence.

\begin{lem}
	\label{lem:AR_guteTeilfolge}Let $\varphi:X \rightarrow \mathbb{R}$ be a ceiling function. Furthermore, let $\bbar A\subset\bbar X$ be a hole such that there exists a $\lambda\in\mathbb{R}_{>0}$ for which
	\[
	\alpha\coloneqq\lim_{n\rightarrow\infty}-\frac{1}{n\lambda}\log\bbar{\mu}\left(\left\{ \left(x,s\right)\in\bbar X\middle|\; \forall\tau\in\left[0,n\lambda\right]:\Phi_{\tau}\left(x,s\right)\notin\bbar A\right\} \right)
	\]
	 exists. Then for any sequence $\left(t_{k}\right)_{k\in\mathbb{N}}\rightarrow\infty$ of positive real numbers the limit 
	\[
	\lim_{k\rightarrow\infty}-\frac{1}{t_{k}}\log\bbar{\mu}\left(\left\{ \left(x,s\right)\in\bbar X\middle|\; \forall\tau\in\left[0,t_{k}\right]:\Phi_{\tau}\left(x,s\right)\notin\bbar A\right\} \right)
	\]
	 exists and coincides with $\alpha$. Hence, the escape rate $\rho\left(\bbar A,\varphi\right)$ exists and equals $\alpha$.
\end{lem}
\begin{proof}
	Define the sequence $\left(m_{k}\right)_{k\in\mathbb{N}}$ by $m_{k}\coloneqq\left\lceil t_{k}/\lambda\right\rceil$. Then
	\begin{align*}
	 & -\frac{1}{m_{k}\lambda}\log\bbar{\mu}\left(\left\{ \left(x,s\right)\in\bbar X\middle|\; \forall\tau\in\left[0,\left(m_{k}-1\right)\cdot\lambda\right]:\Phi_{\tau}\left(x,s\right)\notin\bbar A\right\} \right)\\
	\leq & -\frac{1}{t_{k}}\log\bbar{\mu}\left(\left\{ \left(x,s\right)\in\bbar X\middle|\; \forall\tau\in\left[0,t_{k}\right]:\Phi_{\tau}\left(x,s\right)\notin\bbar A\right\} \right)\\
	\leq & -\frac{1}{\left(m_{k}-1\right)\cdot\lambda}\log\bbar{\mu}\left(\left\{ \left(x,s\right)\in\bbar X\middle|\; \forall\tau\in\left[0,m_{k}\lambda\right]:\Phi_{\tau}\left(x,s\right)\notin\bbar A\right\} \right).
	\end{align*}
	Since the sequences given by the lower and upper estimates both converge towards $\alpha$, so does the enclosed sequence.
\end{proof}
\begin{rem}
	\label{rem:AR_zu_DachFunktion1_gleich_Basis}Choosing the step size $\lambda$ as $1$ in Lemma~\ref{lem:AR_guteTeilfolge}, one concludes that the escape rate $\rho\left(A\right)$ for a hole $A\subset X$ with respect to the base transformation and the corresponding escape rate $\rho\left(A,\ind\right)$ for the special flow with ceiling function $\ind$ take the same value.
\end{rem}

The escape rate with respect to the flow can be phrased without explicitly referring to the space $\bbar X$ and its measure $\bbar \mu$.
\begin{lem}
	\label{lem:AR_Nur-NA-Nt}Let $\varphi:X \rightarrow \mathbb{R}$ be a bounded ceiling function. Furthermore, let $A\subset X$ be a hole such that the escape rate $\rho\left(A,\varphi\right)$ exists. Then 
	\begin{align*}
	\rho\left(A,\varphi\right) & =\lim_{t\rightarrow\infty}-\frac{1}{t}\log\bbar{\mu}_{\varphi}\left(\left\{ \left(x,s\right)\in\bbar X_{\varphi}\middle|\; N_{A}\left(x\right)\geq N_{t}^{\varphi}\left(x\right)\right\} \right)\\
	 & =\lim_{t\rightarrow\infty}-\frac{1}{t}\log\mu\left(\left\{ x\in X\middle|\; N_{A}\left(x\right)\geq N_{t}^{\varphi}\left(x\right)\right\} \right).
	\end{align*}
\end{lem}

\begin{proof}
Let $\bbar X_{\varphi},\bbar X_{\ind}$ and $\bbar{\mu}_{\varphi},\bbar{\mu}_{\ind}$
denote the spaces and measures that belong to the flows corresponding to the ceiling functions $\varphi$ and $\ind$ respectively. 

The first equality follows from 
\begin{align*}
\rho\left(A,\varphi\right)
 & =\lim_{t\rightarrow\infty}-\frac{1}{t}\log\bbar{\mu}_{\varphi}\left(\left\{ \left(x,s\right)\in\bbar X_{\varphi}\middle|\; \forall\tau\in\left[0,t\right]:\Phi_{\tau}^{\varphi}\left(x,s\right)\notin \pi_{1,\varphi}^{-1}\left(A\right)\right\} \right)\\
 & =\lim_{t\rightarrow\infty}-\frac{1}{t}\log\bbar{\mu}_{\varphi}\left(\left\{ \left(x,s\right)\in\bbar X_{\varphi}\middle|\; N_{A}\geq N_{t+s}^{\varphi}\right\} \right)\\
 & \leq\lim_{t\rightarrow\infty}-\frac{1}{t}\log\bbar{\mu}_{\varphi}\left(\left\{ \left(x,s\right)\in\bbar X_{\varphi}\middle|\; N_{A}\geq N_{t+\sup\varphi}^{\varphi}\right\} \right)\\
 & =\lim_{t\rightarrow\infty}-\frac{1}{t-\sup\varphi}\bbar{\mu}_{\varphi}\left(\left\{ \left(x,s\right)\in\bbar X_{\varphi}\middle|\; N_{A}\geq N_{t}^{\varphi}\right\} \right)\\
 & =\lim_{t\rightarrow\infty}-\frac{1}{t}\log\bbar{\mu}_{\varphi}\left(\left\{ \left(x,s\right)\in\bbar X_{\varphi}\middle|\; N_{A}\geq N_{t}^{\varphi}\right\} \right)\\
 & \leq\lim_{t\rightarrow\infty}-\frac{1}{t}\log\bbar{\mu}_{\varphi}\left(\left\{ \left(x,s\right)\in\bbar X_{\varphi}\middle|\; N_{A}\geq N_{t+s}^{\varphi}\right\} \right)\\
 & =\rho\left(A,\varphi\right).
\end{align*}
 The second equality is a consequence of 
\begin{align*}
\rho\left(A,\varphi\right) & =\lim_{t\rightarrow\infty}-\frac{1}{t}\log\bbar{\mu}_{\varphi}\left(\left\{ \left(x,s\right)\in\bbar X_{\varphi}\middle|\; N_{A}\geq N_{t}^{\varphi}\right\} \right)\\
 & \leq\lim_{t\rightarrow\infty}-\frac{1}{t}\log\left(\inf\varphi\cdot\bbar{\mu}_{\ind}\left(\left\{ \left(x,s\right)\in\bbar X_{\ind}\middle|\; N_{A}\geq N_{t}^{\varphi}\right\} \right)\right)\\
 & =\lim_{t\rightarrow\infty}-\frac{1}{t}\log\left(\bbar{\mu}_{\ind}\left(\left\{ \left(x,s\right)\in\bbar X_{\ind}\middle|\; N_{A}\geq N_{t}^{\varphi}\right\} \right)\right)\\
 & =\lim_{t\rightarrow\infty}-\frac{1}{t}\log\left(\sup\varphi\cdot\bbar{\mu}_{\ind}\left(\left\{ \left(x,s\right)\in\bbar X_{\ind}\middle|\; N_{A}\geq N_{t}^{\varphi}\right\} \right)\right)\\
 & \leq\lim_{t\rightarrow\infty}-\frac{1}{t}\log\left(\bbar{\mu}_{\varphi}\left(\left\{ \left(x,s\right)\in\bbar X_{\varphi}\middle|\; N_{A}\geq N_{t}^{\varphi}\right\} \right)\right)\\
 & =\rho\left(A,\varphi\right)
\end{align*}
 and 
\begin{align*}
 \bbar{\mu}_{\ind}\left(\left\{ \left(x,s\right)\in\bbar X_{\ind}\middle|\; N_{A}\geq N_{t}^{\varphi}\right\} \right)
= \mu\left(\left\{ x\in X\middle|\; N_{A}\geq N_{t}^{\varphi}\right\} \right).
\end{align*}

\end{proof}

Whereas Proposition~\ref{prop:EigenschaftenAR_Mengen} describes how altering the hole affects the escape rate, the following proposition describes basic properties of the escape rate's dependency on the ceiling function.
\begin{prop}\label{prop:EigenschaftenAR_Dach}
	Let $\varphi,\psi:X\rightarrow\mathbb{R}$ be bounded ceiling functions. Furthermore, fix a hole $A\subset X$ for which the escape rates $\rho\left(A,\varphi\right)$, $\rho\left(A,\psi\right)$ exist.
	\begin{enumerate}
	\item \label{enu:EigenschaftenAR_Dach_Monotonie}If $\varphi\leq\psi$, then $\rho\left(A,\varphi\right)\geq\rho\left(A,\psi\right)$.
	\item \label{enu:EigenschaftenAR_Dach_Skalar}For all $\lambda\in\mathbb{R}_{>0}$ one has $\rho\left(A,\lambda\varphi\right)=\lambda^{-1}\rho\left(A,\varphi\right)$.
	\item \label{enu:EigenschaftenAR_Dach_Coboundary}If there is a bounded, measurable function  $\chi:X\rightarrow\mathbb{R}$ such that $\psi=\varphi+\chi\circ\theta-\chi$, then $\rho\left(A,\psi\right)=\rho\left(A,\varphi\right)$.
	\end{enumerate}
\end{prop}
\begin{proof}
Using Lemma~\ref{lem:AR_Nur-NA-Nt}, property \eqref{enu:EigenschaftenAR_Dach_Monotonie} follows from the fact that $\varphi\leq\psi$ implies $N_{t}^{\psi}\leq N_{t}^{\varphi}$. Property  \eqref{enu:EigenschaftenAR_Dach_Skalar} is a consequence of  $N_{t}^{\lambda\varphi}=N_{{t}/{\lambda}}^{\varphi}$. For \eqref{enu:EigenschaftenAR_Dach_Coboundary} note that $S_{n}\psi=S_{n}\varphi+\chi\circ\theta^{n}-\chi$ and thus $\left|S_{n}\psi-S_{n}\varphi\right|\leq\left|\chi\circ\theta^{n}-\chi\right|\leq2\cdot\sup\left|\chi\right|\eqqcolon c<\infty$. This implies 
\[
N_{t}^{\varphi}=\min\left\{ n\in\mathbb{N}\middle|\; S_{n}\varphi>t\right\} \leq\min\left\{ n\in\mathbb{N}\middle|\; S_{n}\psi-c>t\right\} =N_{t+c}^{\psi},
\]
 from which one concludes that
\begin{align*}
\rho\left(\varphi\right) & =\lim_{t\rightarrow\infty}-\frac{1}{t}\log\left(\mu\left(\left\{ x\in X\middle|\; N_{A}\left(x\right)\geq N_{t}^{\varphi}\left(x\right)\right\} \right)\right)\\
 & \leq\lim_{t\rightarrow\infty}-\frac{1}{t}\log\left(\mu\left(\left\{ x\in X\middle|\; N_{A}\left(x\right)\geq N_{t+c}^{\psi}\left(x\right)\right\} \right)\right)\\
 & =\lim_{t\rightarrow\infty}-\frac{1}{t}\log\left(\mu\left(\left\{ x\in X\middle|\; N_{A}\left(x\right)\geq N_{t}^{\psi}\left(x\right)\right\} \right)\right)\\
 & =\rho\left(\psi\right).
\end{align*}
 Analogously one obtains $N_{t}^{\psi}\leq N_{t+c}^{\varphi}$ and thus $\rho\left(\varphi\right)\geq\rho\left(\psi\right)$, yielding the desired equality.
\end{proof}

Property \eqref{enu:EigenschaftenAR_Dach_Coboundary} means that the escape rate is invariant under changes of the ceiling function by addition of a \emph{coboundary} $\chi\circ\theta-\chi$. This is not surprising, considering that (invertible) special flows over the same base transformation are isomorphic if their ceiling functions differ only by a coboundary \cite[Theorem 1]{Gurevich:1965}.

\begin{prop}\label{prop:AR_glmKonvergenz}
	The escape rate with respect to a fixed hole is continuous when considered as a function on the space of ceiling functions equipped with the supremum norm.
\end{prop}

\begin{proof}
Let $A \subset X$ be a hole and let $\varphi, \psi: X \rightarrow \mathbb{R}$ be ceiling functions. Note that this implies $\inf \varphi, \inf \psi > 0$. The result follows from the observation
\begin{equation}
	\left(1-\frac{\left\Vert\varphi-\psi\right\Vert_\infty}{\inf\varphi}\right)\cdot\rho\left(\varphi\right)
	\leq
	\rho\left(\psi\right)
	\leq
	\left(1+\frac{\left\Vert\varphi-\psi\right\Vert_\infty}{\inf\varphi}\right)\cdot\rho\left(\varphi\right)
	\label{eq:AR-Approximation-Abschaetzung}
\end{equation}
 and the scaling property of the escape rate (Proposition~\ref{prop:EigenschaftenAR_Dach} (\ref{enu:EigenschaftenAR_Dach_Skalar})).
\end{proof}

\begin{rem}
	\label{rem:EigenschaftenAR_Dach_Oben-Unten-Bemerkung}In Propositions \ref{prop:EigenschaftenAR_Mengen}, \ref{prop:EigenschaftenAR_Dach} and \ref{prop:AR_glmKonvergenz} the escape rate can be replaced by the upper or lower escape rate to deal with situations where the two do not coincide.
\end{rem}

\section{The local escape rate}\label{chap:LokaleAR}

In this section it shall be assumed that $\left(X,d\right)$ is a metric space with a finite Borel measure $\mu$. Let $\theta:\left(X,\mu\right)\rightarrow\left(X,\mu\right)$ be an endomorphism and let $\Phi$ denote the semiflow with base transformation $\theta$ under a bounded ceiling function $\varphi:X\rightarrow\mathbb{R}$. The invariant measure of the flow is denoted by $\bbar \mu$.

\begin{defn}
	\label{def:lokale-AR}Let $x\in X$ be a point such that every neighbourhood $B_{r}\left(x\right)\subset X$, $r\in\mathbb{R}_{>0}$ is a hole.  
   	If the limit 
	\[
	\rho\left(x,\varphi\right)\coloneqq\lim_{r\rightarrow0}\frac{\rho\left(B_{r}\left(x\right),\varphi\right)}{\mu\left(B_{r}\left(x\right)\right)}
	\] exists, it is called the \emph{local escape rate} in $x$ with respect to the ceiling function $\varphi$. If $\varphi=\ind$, then this is the same as the local escape rate in $x$ of the base transformation.
\end{defn}

\begin{rem}
	\label{rem:LokaleAR_Normierung}Using $\mu\left(B_{r}\left(x\right)\right)$ for normalisation is a sensible choice, because the escape rate only depends on the shadow  of the hole.
\end{rem}

Amongst others, the local escape rate for discrete time systems has been investigated in \cite{Keller:2009,Bunimovich:2011,Ferguson:2012,Cristadoro:2013}. Theorem~\ref{thm:LokaleAR_IntegralMalBasis} characterises how the local escape rate with respect to the semiflow is connected to that of the base transformation. The following heuristic motivation of this result provides a blueprint for the proof.

According to Lemma~\ref{lem:AR_Nur-NA-Nt} the escape rate for a hole $A\subset X$ can be written as 
\[
\rho\left(A,\varphi\right)=\lim_{t\rightarrow\infty}-\frac{1}{t}\log\left(\mu\left(\left\{ x\in X\middle|\; N_{A}\geq N_{t}\right\} \right)\right).
\]
 Assuming that $S_{n}\varphi/n\rightarrow\mu\left(\varphi\right)$ holds and that this convergence happens sufficiently fast, $N_{t}$ can be approximated by $S_{N_{t}}\varphi/\mu\left(\varphi\right)$. Since $\varphi$ is assumed to be bounded, $S_{N_{t}}\varphi$ provides an approximation of $t$. Together, one obtains the naive approximation 
\[
\lim_{t\rightarrow\infty}-\frac{1}{t}\log\bbar{\mu}\left(\left\{ x\in X\middle|\; N_{A}\left(x\right)\geq\frac{t}{\mu\left(\varphi\right)}\right\} \right)=\frac{1}{\mu\left(\varphi\right)}\rho\left(A,\ind\right)
\]
 for the escape rate. The escape rate is blind to the values of the ceiling function on $A$, but $\mu\left(\varphi\right)$ contains this information. Therefore the scaling with $1 / \mu\left(\varphi\right)$ can only hold in the limit for shrinking holes -- which is exactly the situation that the local escape rate describes.

The condition that $S_{n}\varphi/n\rightarrow\mu\left(\varphi\right)$ converges sufficiently fast can be made explicit by requiring an exponential decay of  
\begin{equation}
	\label{eq:P_epsilon_k}
	P_{k}^{\varepsilon}\coloneqq\mu\left(\left\{ x\in X\middle|\; \sup_{l\geq k}\left|\frac{S_{l}\varphi\left(x\right)}{l}-\mu\left(\varphi\right)\right|\geq\varepsilon\right\} \right).
\end{equation}

\begin{thm}\label{thm:LokaleAR_IntegralMalBasis}
	Let $x_{0}\in X$ be a point for which the local escape rate with respect to the base exists and where $\mu$ does not have an atom. Let $P_{k}^{\varepsilon}$ be as in equation~\eqref{eq:P_epsilon_k} and assume that 
	\[
	\forall\varepsilon>0\exists C>0,\zeta<1:\forall k\in\mathbb{N}:\quad P_{k}^{\varepsilon}\leq C\zeta^{k}.
	\]
	 Then 
	\[
	\rho\left(x_{0},\varphi\right)=\frac{\rho\left(x_{0},\ind\right)}{\mu\left(\varphi\right)}.
	\]
\end{thm}

\begin{proof}
If $\rho\left(x_{0},\ind\right)$ is zero or infinity, then the theorem holds, because 
\begin{align*}
\rho\left(x_{0},\ind\right)\cdot\frac{1}{\sup\varphi}  =\rho\left(x_{0},\sup\varphi\right)
  & \leq\rho\left(x_{0},\varphi\right) \\
 & \leq\rho\left(x_{0},\inf\varphi\right)=\rho\left(x_{0},\ind\right)\cdot\frac{1}{\inf\varphi}.
\end{align*}

Now consider the case $0 < \rho\left(x_{0},\ind\right) < \infty$. The general strategy of this proof is to obtain upper and lower estimates for  
\[
\mu\left(\left\{ x\in X\middle|\; N_{B_{r}}\geq N_{t}\right\} \right)
\]
which provide lower and upper estimates for the local escape rate. In order to keep the notation as compact as possible, $B_{r}\left(x_{0}\right)$ is shortened to $B_{r}$.

Rewrite $N_{t}$ as  
\[
N_{t}=\frac{S_{N_{t}}\varphi}{\mu\left(\varphi\right)+R_{N_{t}}}
\]
with 
\[
R_{N_{t}}\coloneqq\frac{S_{N_{t}}\varphi}{N_{t}}-\mu\left(\varphi\right)
\]
 and note that $\sup\varphi<\infty$ implies $N_{t}\left(x\right) \rightarrow \infty$ uniformly as $t\rightarrow\infty$.

The objective is now to calculate an upper bound for the local escape rate.

Fix an $\varepsilon>0$ with $\varepsilon<\mu\left(\varphi\right)$.
\begin{align}
	 & \mu\left(\left\{ x\in X\middle|\; N_{B_{r}}\geq N_{t}\right\} \right)\label{eq:AbschaetzungNachUnten_Stern1_P}\\
	=\, & \mu\left(\left\{ x\in X\middle|\; N_{B_{r}}\geq\frac{S_{N_{t}}\varphi}{\mu\left(\varphi\right)+R_{N_{t}}}\right\} \right)\nonumber \\
	\geq\, & \mu\left(\left\{ x\in X\middle|\; N_{B_{r}}\geq\frac{S_{N_{t}}\varphi}{\mu\left(\varphi\right)-\varepsilon}\right\} \setminus\left\{ x\in X\middle|\; R_{N_{t}}<-\varepsilon\right\} \right)\nonumber \\
	\geq\, & \mu\left(\left\{ x\in X\middle|\; N_{B_{r}}\geq\frac{S_{N_{t}}\varphi}{\mu\left(\varphi\right)-\varepsilon}\right\} \right)-\underbrace{\mu\left(\left\{ x\in X\middle|\; R_{N_{t}}<-\varepsilon\right\} \right)}\nonumber \\
	\geq\, & \underset{\eqqcolon\left(\star\right)}{\underbrace{\mu\left(\left\{ x\in X\middle|\; N_{B_{r}}\geq\frac{t+\sup\varphi}{\mu\left(\varphi\right)-\varepsilon}\right\} \right)}}-P_{N_{t}}^{\varepsilon}\nonumber 
\end{align}
 $P_{N_{t}}^{\varepsilon}$ can be bounded from above via
\begin{equation}
P_{N_{t}}^{\varepsilon}\leq C\zeta^{\left(\frac{t}{\sup\varphi}\right)}=C\zeta_{1}^{t}\label{eq:AbschaetzungPnachOben}
\end{equation}
 using $\zeta_{1}\coloneqq\zeta^{1/\sup\varphi}<1$, because $N_{t}\left(x\right)\geq t/\sup\varphi$. Since the local escape rate with respect to the base transformation is assumed to be finite and positive, one has
\[
\forall\varepsilon'>0\exists r_{0}\left(\varepsilon'\right):\forall r\in\left(0,r_{0}\right):\quad\frac{\rho\left(B_{r},\ind\right)}{\mu\left(B_{r}\right)}\leq\left(1+\varepsilon'\right)\rho\left(x_{0},\ind\right).
\]
 This implies 
\begin{equation}
	\rho\left(B_{r},\ind\right)\leq\left(1+\varepsilon'\right)\rho\left(x_{0},\ind\right)\mu\left(B_{r}\right)\label{eq:Abschaetzung_ObereAR-kleinergleich}
\end{equation}
for sufficiently small $r$. Since the escape rate with respect to the ceiling function $\ind$ is equal to the escape rate with respect to the base transformation, Remark~\ref{rem:ARdiskret_reellerParameter} allows one to conclude that
\begin{equation*}
	\forall\varepsilon''>0\exists t_{0}\left(\varepsilon'',r\right):\forall t\geq t_{0}:
	-\frac{1}{t}\log\mu\left(\left\{ x\in X\middle|\; N_{B_{r}}\geq t\right\} \right)\leq\left(1+\varepsilon''\right)\rho\left(B_{r},\ind\right).
\end{equation*}
This implies the estimate  
\begin{equation}
	\mu\left(\left\{ x\in X\middle|\; N_{B_{r}}\geq t\right\} \right)\geq\left(e^{-\left(1+\varepsilon''\right)\rho\left(B_{r},\ind\right)}\right)^{t}\label{eq:Abschaetzung_ARmu-groessergleich}
\end{equation}
 for sufficiently large $t$. Inequalities \eqref{eq:Abschaetzung_ObereAR-kleinergleich} and \eqref{eq:Abschaetzung_ARmu-groessergleich} imply 
\[
\mu\left(\left\{ x\in X\middle|\; N_{B_{r}}\geq t\right\} \right)\geq\left(e^{-\left(1+\varepsilon''\right)\left(1+\varepsilon'\right)\rho\left(x_{0},\ind\right)\mu\left(B_{r}\right)}\right)^{t}\eqqcolon\zeta_{2}^{t}
\]
 with $\zeta_{2}\left(r,\varepsilon',\varepsilon''\right)<1$ for sufficiently small $r$ and sufficiently large $t\left(r\right)$. Note that $\zeta_{2}$ tends to $1$ for $r\rightarrow0$ because of  $\lim_{r\rightarrow0}\mu\left(B_{r}\left(x_{0}\right)\right)=0$ which in turn follows from $\mu$ being finite and without atom in $x_{0}$. This new inequality can be used in conjunction with \eqref{eq:AbschaetzungPnachOben} to continue the chain of inequalities \eqref{eq:AbschaetzungNachUnten_Stern1_P}:
\begin{align*}
	\mu\left(\left\{ x\in X\middle|\; N_{B_{r}}\geq N_{t}\right\} \right) & \geq\left(\star\right)-P_{N_{t}}^{\varepsilon}\\
	 & \geq\zeta_{2}^{\frac{t+\sup\varphi}{\mu\left(\varphi\right)-\varepsilon}}-C\zeta_{1}^{t}\\
	 & \geq\zeta_{2}^{\frac{t+\sup\varphi}{\mu\left(\varphi\right)-\varepsilon}}\cdot\left(1-C\zeta_{2}^{-\left(\frac{1+\sup\varphi\cdot t^{-1}}{\mu\left(\varphi\right)-\varepsilon}\right)\cdot t}\zeta_{1}^{t}\right)\\
	 & =\zeta_{2}^{\frac{t+\sup\varphi}{\mu\left(\varphi\right)-\varepsilon}}\cdot\left(1-C\zeta_{3}^{t}\right)
\end{align*}
with  
\[
\zeta_{3}\coloneqq\zeta_{2}^{-\frac{1+\sup\varphi\cdot t^{-1}}{\mu\left(\varphi\right)-\varepsilon}}\zeta_{1}.
\]
This constitutes the necessary lower estimate for $\mu\left(\left\{ x\in X\middle|\; N_{B_{r}}\geq N_{t}\right\} \right)$ which implies
\begin{equation*}
-\frac{1}{t}\log\mu\left(\left\{ x\in X\middle|\; N_{B_{r}}\geq N_{t}\right\} \right) 
\leq-\frac{1}{t}\cdot\frac{t+\sup\varphi}{\mu\left(\varphi\right)-\varepsilon}\log\zeta_{2}-\frac{1}{t}\log\left(1-C\zeta_{3}^{t}\right).
\end{equation*}
 From $\zeta_{1} < 1$ and the previously remarked property of $\zeta_{2}$ one concludes that $\zeta_{3}<1$ for sufficiently small $r$ and sufficiently large $t\left(r\right)$. This implies that $1-C\zeta_{3}^{t}$ is positive, bounded and bounded away from zero for $t\rightarrow\infty$ and sufficiently small $r$. Consequently, 
\[
\lim_{t\rightarrow\infty}\frac{1}{t}\log\left(1-C\zeta_{3}^{t}\right)=0.
\]
 Utilising this property and plugging in the definition of $\zeta_{2}$, one obtains 
\begin{align*}
	\rho\left(B_{r},\varphi\right) & =\lim_{t\rightarrow\infty}-\frac{1}{t}\log\mu\left(\left\{ x\in X\middle|\; N_{B_{r}}\geq N_{t}\right\} \right)\\
	 & \leq\frac{1}{\mu\left(\varphi\right)-\varepsilon}\cdot\left(1+\varepsilon''\right)\left(1+\varepsilon'\right)\rho\left(x_{0},\ind\right)\mu\left(B_{r}\right)+0
\end{align*}
 for sufficiently small $r$. This inequality holds for all $\varepsilon''>0$ and thus one concludes  
\[
\frac{\rho\left(B_{r},\varphi\right)}{\mu\left(B_{r}\right)}\leq\frac{1}{\mu\left(\varphi\right)-\varepsilon}\cdot\left(1+\varepsilon'\right)\rho\left(x_{0},\ind\right).
\]
 Letting $r$ tend to $0$, $\varepsilon'$ is allowed to become arbitrarily small, therefore  
\[
\limsup_{r\rightarrow\infty}\frac{\rho\left(B_{r},\varphi\right)}{\mu\left(B_{r}\right)}\leq\frac{1}{\mu\left(\varphi\right)-\varepsilon}\cdot\rho\left(x_{0},\ind\right).
\]
 Finally, $\varepsilon>0$ may be chosen arbitrarily small and thus one obtains the desired upper estimate for the local escape rate:
\[
\limsup_{r\rightarrow\infty}\frac{\rho\left(B_{r},\varphi\right)}{\mu\left(B_{r}\right)}\leq\frac{\rho\left(x_{0},\ind\right)}{\mu\left(\varphi\right)}.
\]

The lower estimate 
\[
\liminf_{r\rightarrow\infty}\frac{\rho\left(B_{r},\varphi\right)}{\mu\left(B_{r}\right)}\geq\frac{\rho\left(x_{0},\ind\right)}{\mu\left(\varphi\right)}
\]
can be obtained by an analogous calculation. Both estimates together imply the theorem.
\end{proof}

In \cite{Dreher:2015}, it was remarked that Theorem~\ref{thm:LokaleAR_IntegralMalBasis} holds for Markov shifts on a finite alphabet and continuous ceiling functions. This was achieved by using a result from \cite{Katz:1961} to obtain the result for cylinder functions and then applying the continuity of the escape rate (Proposition~\ref{prop:AR_glmKonvergenz}).

The following section improves on that result by showing that the condition on $P_{k}^{\varepsilon}$ that is required by Theorem~\ref{thm:LokaleAR_IntegralMalBasis} holds for $g$-measures with a differentiability property of the associated pressure function (Proposition~\ref{prop:deviationProperty}).
 
\section{Exponential decay for large deviation bounds for $g$-measures }\label{sec:weakThermo}

Let $\left(X,\theta\right)$ be a (onesided) subshift of finite type
over a finite alphabet, denote its Borel $\sigma$-algebra by $\mathcal{B}$
and let $\mathcal{C}(X)$ be the set of real-valued, continuous functions
on $X$. The set of all probability measures on $X$ is denoted by
$\mathcal{M}(X)$.  The \emph{pressure} \emph{function} $P:\mathcal{C}(X)\to\mathbb{R}$
is defined as 
\begin{equation}
P(f):=\lim_{n\to\infty}\frac{1}{n}\log\sum_{C\in C_{n}}\exp\left(\sup_{x\in C}S_{n}f(x)\right).\label{Def:pressure}
\end{equation}
 
\begin{lem}
Let $\psi,\varphi\in\mathcal{C}(X)$ with $\psi<0$. Then the map
$p:\mathbb{R}\to\mathbb{R}$, $s\mapsto P(\varphi+s\psi)$, is a convex
function, which is decreasing from $\infty$ to $-\infty$.\end{lem}
\begin{proof}
Since the pressure is convex and increasing as a function of $\mathcal{C}(X)$,
the function $p$ is convex, and increasing \cite{Walters:2000}.
The divergence property follows from the inequality $P(0)+\inf f\leq P(f)\leq P(0)+\sup f$,
which is valid for all $f\in\mathcal{C}(X)$ (cf. \cite[Theorem 9.7.(ii)+(v)]{Walters:2000}). 
\end{proof}
In this setting, the Perron-Frobenius operator $\mathcal{L}_{\psi}:\mathcal{C}(X)\to\mathcal{C}(X)$
can be defined for $\psi\in\mathcal{C}(X)$ by
\[
\mathcal{L}_{\psi}g(x):=\sum_{y\in\theta^{-1}\left\{ x\right\} }\exp\left(\psi(y)\right)\,g(y).
\]
\begin{fact}
\label{Fact:Eigen}By the Schauder-Tychonov fixed point theorem, for
any $\psi\in\mathcal{C}\left(X\right)$ there exists $\mu\in\mathcal{M}(X)$
and $\lambda>0$ such that 
\begin{equation}
\mathcal{L}_{\psi}^{*}\mu=\lambda\mu,\label{Formel:eigenmeasure}
\end{equation}
where $\mathcal{L}_{\psi}^{*}$ denotes the adjoint operator of $\mathcal{L}_{\psi}$.
Moreover, the eigenvalue $\lambda$ is uniquely determined by $\psi$
and is related to the pressure by $\lambda=\exp P\left(\psi\right)$.
\end{fact}
Following \cite{Yuri:1998} and \cite{MR1819804}, any $\mu\in\mathcal{M}(X)$
fulfilling (\ref{Formel:eigenmeasure}) for $\psi\in\mathcal{C}\left(X\right)$
and $\lambda=\exp P(\psi)$ is called a \emph{weak Gibbs measure}
for $\psi$. 
\begin{fact}
[\cite{MR1819804}]\label{lemma:weakGibbsProperty}Let $\mu\in\mathcal{M}(X)$
be a weak Gibbs measure for $\psi\in\mathcal{C}\left(X\right)$. Then
for all $n\in\mathbb{N}$ and $x\in\left[x_{1},\ldots,x_{n}\right]$ one has that
\[
\exp(-\eta_{\psi}(n))\leq\frac{\mu\left(\left[x_{1},\ldots,x_{n}\right]\right)}{\exp\left(S_{n}\psi(x)-nP(\psi)\right)}\leq\exp(\eta_{\psi}(n)).
\]
In here, $\eta_{\psi}(n):=\var_{n}\left(S_{n}\psi\right)$ with $\var_{n}(\psi):=\sup_{C\in C_{n}}\sup_{x,y\in C}\left\{ \left|\psi(x)-\psi(y)\right|\right\} $,
$n\in\mathbb{N}$, and it holds that $\eta_{\psi}(n)/n\to0$ as $n\to\infty$.
From these inequalities one deduces immediately that $\sup\psi<P(\psi)$
is sufficient for $\mu$ being atom-free and positive on all non-empty
open sets.
\end{fact}
If additionally $\psi\in\mathcal{C}(X)$ is a normalised potential,
i.e. $\mathcal{L}_{\psi}\1=\1$, then $P\left(\psi\right)=0$ and
any weak Gibbs measure $\mu\in\mathcal{M}(X)$ for $\psi$ is called
a \emph{$g$-measure} for $\psi$. The notion of $g$-measures was
introduced by M. Keane in \cite{Keane:72}. Since one necessarily
has $P\left(\psi\right)=0>\psi$, it follows from Lemma \ref{lemma:weakGibbsProperty}
that $\mu$ is supported on $X$ and non-atomic. The measure $\mu$
is also $\theta$-invariant, i.e. $\mu\circ\theta^{-1}=\mu$, since
for all $B\in\mathcal{B}$ one has  
\[
\mu\circ\theta^{-1}\left(B\right)=\int\1_{B}\circ\theta\:\mbox{d}\mu=\int\mathcal{L}_{\varphi}\left(\1_{B}\circ\theta\right)\:\mbox{d}\mu=\int\1_{B}\cdot\mathcal{L}_{\varphi}\left(\1\right)\:\mbox{d}\mu=\mu\left(B\right).
\]
 It is also characterised as an equilibrium measure (cf. \cite{Ledrappier:1974}),
i.e.
\[
P(\varphi)=\sup\left\{ h_{\eta}+\eta(\varphi):\eta\in\mathcal{M}_{\theta}(X)\right\} =h_{\mu}+\mu(\varphi),
\]
 where $\mathcal{M}_{\theta}(X)$ are the $\theta$-invariant probability
measures on $X$ and $h_{\eta}$ denotes the measure theoretical entropy
(see e.g. \cite{Walters:2000} for the definition). 

If $\psi\in\mathcal{C}\left(X\right)$ with $P\left(\psi\right)=0$
is H{\"o}lder continuous then by the Ruelle-Perron-Frobenius theory
one knows that there exists a $\widetilde{\psi}\in\mathcal{C}\left(X\right)$
cohomologous to $\psi$ such that $\mathcal{L}_{\widetilde{\psi}}\1=\1$.
In fact, $\widetilde{\psi}:=\psi+\log h_{\psi}-\log\left(h_{\psi}\circ\theta\right)$,
where $h_{\psi}$ is a (positive) eigenfunction of $\mathcal{L}_{\psi}$
for the eigenvalue $1$. Hence,  the unique Gibbs measure for the potential
$\psi$ is also a $g$-measure for $\widetilde{\psi}$. 

Next, one defines the \emph{free Helmholtz energy }for the pair $\left(\varphi,\mu\right)$
with $\varphi\in\mathcal{C}\left(X\right)$ and $\mu$ a $g$-measure
for $\psi$ by 
\[
H_{n}(t):=\frac{1}{n}\log\int\exp\left(t\,S_{n}\varphi\right)\,d\mu,
\]
 which is obviously finite for any $n\in\mathbb{N}$ and $t\in\mathbb{R}$.
The following proposition guarantees that $H(t):=\lim_{R\to\infty}H_{R}(t)$
exists and is finite for all $t\in\mathbb{R}$. Thus, by H{\"o}lder's
inequality one observes that $H_{n}$, $n\in\mathbb{N}$, and consequently
also $H$ are convex finite functions. Indeed, the free energy and
the pressure function are related as follows. 
\begin{prop}
[\cite{MR1819804}] \label{Thm:freeEnergy} Let $\varphi,\psi\in\mathcal{C}\left(X\right)$,
$\mathcal{L}_{\psi}\1=\1$ and let $\mu\in\mathcal{M}(X)$ be a $g$-measure
for the potential $\psi$. The free Helmholtz energy $H$ is determined
by the pressure function via 
\[
H(t)=P\left(t\varphi+\psi\right).
\]

\end{prop}
Also, by the convexity of the pressure function and the fact that
$\mu$ is an equilibrium measure, one finds that $\mu\left(\varphi\right)\in\nabla H\left(0\right)$,
where $\nabla H\left(0\right)$ denotes the sub-differential of $H$
in $0$. In particular, if $H$ is differentiable in $0$ then $H'\left(0\right)=p'\left(0\right)=\mu\left(\varphi\right)$. 
\begin{thm}
[Theorem II.6.3 of \cite{Ellis:1985}] \label{Theorem-Ellis}The following
statements are equivalent: 
\begin{itemize}
\item [{\rm (a)}] $H$ is differentiable in $0$ with $H'(0)=M$.
\item [{\rm (b)}] For every $\varepsilon>0$ there exists a number $\delta(\varepsilon)>0$
such that for all $n$ sufficiently large
\[
\mu\left(\left\{ x\in X\middle|\; \left|\frac{S_{n}\varphi\left(x\right)}{n}-M\right|\geq\varepsilon\right\} \right)\leq\exp(-n\delta(\varepsilon)).
\]

\end{itemize}
\end{thm}
Combining Proposition~\ref{Thm:freeEnergy} and Theorem~\ref{Theorem-Ellis},
one obtains the main result of this section which implies the exponential decay property for
$P_{k}^{\varepsilon}$ that is required in Theorem~\ref{thm:LokaleAR_IntegralMalBasis}.
\begin{prop}\label{prop:deviationProperty}
Let $\varphi,\psi\in\mathcal{C}\left(X\right)$, $\mathcal{L}_{\psi}\1=\1$
and let $\mu\in\mathcal{M}(X)$ be a $g$-measure for the potential
$\psi$ and suppose that $t\mapsto P\left(t\varphi+\psi\right)$ is
differentiable in $0$. Then for every $\varepsilon>0$ there exists
a number $\delta(\varepsilon)>0$ such that for $n$ sufficiently
large
\[
\mu\left(\left\{ x\in X\middle|\; \left|\frac{S_{n}\varphi\left(x\right)}{n}-\mu\left(\varphi\right)\right|\geq\varepsilon\right\} \right)\leq\exp(-n\delta(\varepsilon)).
\]
\end{prop}
 
 \section{The escape rate as an induced pressure}\label{chap:AR-als-induzierter-Druck}

The escape rate can be formulated as an induced pressure as defined by Jaerisch, Kesseb\"ohmer and Lamei \cite{Jaerisch:2014}. The following definition mirrors \cite[Definition 1.1]{Jaerisch:2014}. Please note that the notation deviates slightly from that in \cite{Jaerisch:2014} in order to accommodate for the fact that $\varphi$ already denotes the ceiling function.

\begin{defn}
	Let $\left(X,\theta\right)$ with $X\coloneqq S_{P}^{\mathbb{N}}$ be a onesided subshift on an alphabet $S$ with transition matrix $P$. For functions $p,\varphi:X\rightarrow\mathbb{R}$ with $\varphi\geq0$ and a set $\mathcal{C}\subset\Sigma_{P}^{*}$, the \emph{$\varphi$-induced pressure of $p$} (with respect to $\mathcal{C}$) for $\eta>0$ is defined as 
	\begin{equation}
	\mathcal{P}_{\varphi}\left(p,\mathcal{C}\right)\coloneqq\limsup_{t\rightarrow\infty}\frac{1}{t}\cdot\log\sum_{\substack{w\in\mathcal{C}\\
	t-\eta<S_{w}\varphi\leq t
	}
	}\exp\left(S_{w}p\right),\label{eq:induzierter-Druck_Definition}
	\end{equation}
	 with $S_{w}p\coloneqq\sup_{x\in\left[w\right]}\sum_{k=0}^{\left|w\right|-1}p\circ\theta^{k}\left(x\right)$, and $\left|w\right|$ denoting the length of the word $w$.
\end{defn}
\begin{rem}
	\label{rem:Induzierter-Druck_unabhaengig-von-eta}According to \cite{Jaerisch:2014}, the induced pressure does not depend on the choice of $\eta>0$.
\end{rem}
In order to phrase the escape rate of a semiflow under a ceiling function $\varphi$ as an induced pressure, it is necessary to choose suitable $p$ and $\mathcal{C}$.
\begin{assumption}
	\label{as:Gibbs-Shift}Let $\left(X,\theta\right)$ with $X\coloneqq S_{P}^{\mathbb{N}}$ be a onesided subshift on an at most countable alphabet $S$ with transition matrix $P$. Let $\mu$ be an invariant probability measure for the subshift for which there is a function 	$p:S_{P}^{\mathbb{N}}\rightarrow\mathbb{R}$ and a non-decreasing sequence of constants $K_{n}>0$ with $\lim_{n\to \infty}\log(K_{n})/n=0$ such that 
	\begin{equation}
	K_{|w|}^{-1}\leq\frac{e^{S_{w}p}}{\mu\left(\left[w\right]\right)}\leq K_{|w|}
	\label{eq:Gibbs-Eigenschaft}
	\end{equation}
	holds for every $w\in\Sigma_{P}^{*}$.
\end{assumption}
\begin{rem}
	Assumption ~\ref{as:Gibbs-Shift} holds in particular for Markov shifts with finite alphabet. In that situation one chooses 
	\[
	p:S_{P}^{\mathbb{N}}\rightarrow\mathbb{R},\quad\left(x_{1},x_{2},\dots\right)\mapsto\log p_{x_{1}x_{2}}
	\]
	and derives a constant $K$ from 
	\[
	e^{S_{w}p}=\prod_{l=1}^{k-1}p_{w_{l}w_{l+1}}\cdot\sup_{s\in S}p_{w_{k}s}=\frac{\sup_{s\in S}p_{w_{k}s}}{\mu\left(\left[w_{1}\right]\right)}\cdot\mu\left(\left[w\right]\right),
	\]
	and the fact that the alphabet is finite.
 \end{rem}
\begin{prop}\label{prop:AR-als-induzierter-Druck}
	Let Assumption~\ref{as:Gibbs-Shift} be fulfilled and assume that there is an $m\in\mathbb{N}$ such that the hole $A\subset X$ can be written as a union of $m$-cylinders. Let $\varphi:X\rightarrow\mathbb{R}$ be a positive, measurable ceiling function that is bounded and bounded away from zero. Choosing $\mathcal{C}$ as  
\begin{align}
\mathcal{C} & \coloneqq\left\{ w\in\Sigma_{P}^{*}\middle|\; \left|w\right|\geq m\,\,\mathrm{and}\,\,\forall0\leq k\leq\left|w\right|-m:\,\theta^{k}\left(\left[w\right]\right)\cap A=\emptyset\right\}, \label{eq:C-zu-Loch-A}
\end{align} 
 and assuming that the escape rate $\rho\left(A,\varphi\right)$ exists, one obtains 
\[
\mathcal{P}_{\varphi}\left(p,\mathcal{C}\right)=-\rho\left(A,\varphi\right).
\]
\end{prop}
\begin{proof}
First note that $\mathcal{C}$ can also be written as 
\begin{align*}
\mathcal{C} & =\left\{ w\in\Sigma_{P}^{*}\middle|\; \left|w\right|\geq m\,\,\mathrm{and}\,\,\forall0\leq k\leq\left|w\right|-m:\,\theta^{k}\left(\left[w\right]\right)\cap A=\emptyset\right\} \\
 & =\left\{ w\in\Sigma_{P}^{*}\middle|\; \left|w\right|\geq m\,\mathrm{and}\,\forall x\in\left[w\right]:\, N_{A}\left(x\right)>\left|w\right|-m\right\} \\
 & =\left\{ w\in\Sigma_{P}^{*}\middle|\; \left|w\right|\geq m\,\mathrm{and}\,\exists x\in\left[w\right]:\, N_{A}\left(x\right)>\left|w\right|-m\right\}. 
\end{align*}
 Using equation~\eqref{eq:Gibbs-Eigenschaft}, one obtains
\begin{align*}
 & \frac{-1}{t}\cdot\log K_{\lceil  t/\inf \varphi\rceil}+\frac{1}{t}\cdot\log\sum_{\substack{w\in\mathcal{C}\\
t-\eta<S_{w}\varphi\leq t
}
}\mu\left(\left[w\right]\right)\\
\leq & \frac{1}{t}\cdot\log\sum_{\substack{w\in\mathcal{C}\\
t-\eta<S_{w}\varphi\leq t
}
}e^{S_{w}p}\\
\leq & \frac{1}{t}\cdot\log K_{\lceil  t/\inf \varphi \rceil}+\frac{1}{t}\cdot\log\sum_{\substack{w\in\mathcal{C}\\
t-\eta<S_{w}\varphi\leq t
}
}\mu\left(\left[w\right]\right).
\end{align*}
 Since the term  $1/t\cdot\log K_{\lceil  t/\inf \varphi \rceil}$ tends to zero for $t\rightarrow\infty$, it can be neglected when looking at the limit. Furthermore, without loss of generality one can assume that $t\geq m\cdot\sup\varphi$ and, due to Remark~\ref{rem:Induzierter-Druck_unabhaengig-von-eta}, that $\eta>\sup\varphi$.
 Under these conditions one obtains
\begin{align*}
  \bigcup_{\substack{w\in\mathcal{C}\\
t-\eta<S_{w}\varphi\leq t
}
}\left[w\right]
= & \left\{ x\in X\middle|\; \exists w\in\Sigma_{P}^{*}:\begin{array}{l}
x\in\left[w\right],\, t-\eta<S_{w}\varphi\leq t,\,\left|w\right|\geq m,\\
\forall y\in\left[w\right]:N_{A}\left(y\right)>\left|w\right|-m
\end{array}\right\} \\
\subset & \left\{ x\in X\middle|\; \exists w\in\Sigma_{P}^{*}:\begin{array}{l}
x\in\left[w\right],\, t-\eta<S_{w}\varphi,\\
\left|w\right|\geq m,\, N_{A}\left(x\right)>\left|w\right|-m
\end{array}\right\} \\
\subset & \left\{ x\in X\middle|\; \exists w\in\Sigma_{P}^{*}:\begin{array}{l}
x\in\left[w\right],\,\left|w\right|\geq N_{t-\eta-\sup\varphi}\left(x\right),\\
\left|w\right|\geq m,\, N_{A}\left(x\right)>\left|w\right|-m
\end{array}\right\} \\
\subset & \left\{ x\in X\middle|\; \exists w\in\Sigma_{P}^{*}:x\in\left[w\right],\, N_{A}\left(x\right)\geq N_{t-\eta-\sup\varphi}\left(x\right)-m\right\} \\
= & \left\{ x\in X\middle|\; N_{A}\left(x\right)\geq N_{t-\eta-\sup\varphi}\left(x\right)-m\right\} \\
\subset & \left\{ x\in X\middle|\; N_{A}\left(x\right)\geq N_{t-\eta-\left(m+1\right)\cdot\sup\varphi}\left(x\right)\right\} 
\end{align*}
 and  
\begin{align*}
  \bigcup_{\substack{w\in\mathcal{C}\\
t-\eta<S_{w}\varphi\leq t
}
}\left[w\right]
= & \left\{ x\in X\middle|\; \exists w\in\Sigma_{P}^{*}:\begin{array}{l}
x\in\left[w\right],\, t-\eta<S_{w}\varphi\leq t,\,\left|w\right|\geq m,\\
\exists y\in\left[w\right]:N_{A}\left(y\right)>\left|w\right|-m
\end{array}\right\} \\
\overset{\left(\star\right)}{\supset} & \left\{ x\in X\middle|\; \exists w\in\Sigma_{P}^{*}:\begin{array}{l}
x\in\left[w\right],\,\left|w\right|=N_{t}\left(x\right)-1,\,\left|w\right|\geq m,\\
\exists y\in\left[w\right]:N_{A}\left(y\right)>\left|w\right|-m
\end{array}\right\} \\
\supset & \left\{ x\in X\middle|\; \exists w\in\Sigma_{P}^{*}:\begin{array}{l}
x\in\left[w\right],\,\left|w\right|=N_{t}\left(x\right)-1,\\
\left|w\right|\geq m,\, N_{A}\left(x\right)>\left|w\right|-m
\end{array}\right\} \\
= & \left\{ x\in X\middle|\; \exists w\in\Sigma_{P}^{*}:\begin{array}{l}
x\in\left[w\right],\,\left|w\right|=N_{t}\left(x\right)-1,\\
N_{t}\left(x\right)-1\geq m,\, N_{A}\left(x\right)\geq N_{t}\left(x\right)-m
\end{array}\right\} \\
\overset{\left(\star\star\right)}{=} & \left\{ x\in X\middle|\; N_{A}\left(x\right)\geq N_{t}\left(x\right)-m\right\} \\
\supset & \left\{ x\in X\middle|\; N_{A}\left(x\right)\geq N_{t-m\cdot\inf\varphi}\left(x\right)\right\} .
\end{align*}
For $\left(\star\right)$ one uses $t-\eta\leq t-\sup\varphi\leq S_{N_{t}-1}\left(x\right)<t$,
 and for $\left(\star\star\right)$ one takes advantage of the fact that $N_{t}\left(x\right)-1\geq m$ is automatically fulfilled due to $t\geq m\cdot\sup\varphi$, and that for each $x\in X$ there is a $w\in\Sigma_{P}^{*}$ with
$\left|w\right|=N_{t}\left(x\right)-1$, so that these two conditions can be omitted.

This leads to the estimates 
\begin{align*}
  \frac{1}{t}\cdot\log\sum_{\substack{w\in\mathcal{C}\\
t-\eta<S_{w}\varphi\leq t
}
}\mu\left(\left[w\right]\right)
\leq & \frac{1}{t}\cdot\log\left(\left\lceil \frac{\eta}{\inf\varphi}\right\rceil \cdot\mu\left(\left\{ x\in X\middle|\; N_{A}\geq N_{t-\eta-\left(m+1\right)\cdot\sup\varphi}\right\} \right)\right)
\end{align*}
 and 
\begin{align*}
  \frac{1}{t}\cdot\log\sum_{\substack{w\in\mathcal{C}\\
t-\eta<S_{w}\varphi\leq t
}
}\mu\left(\left[w\right]\right)
\geq & \frac{1}{t}\cdot\log\mu\left(\left\{ x\in X\middle|\; N_{A}\left(x\right)\geq N_{t-m\cdot\inf\varphi}\left(x\right)\right\} \right),
\end{align*}
 by noting that the sum counts a point in the union at least once and at most $\left\lceil \eta/\inf\varphi\right\rceil $ times. An application of Lemma~\ref{lem:AR_Nur-NA-Nt} gives that $-\rho\left(A,\varphi\right)$ is both an upper and a lower bound for  $\mathcal{P}_{\varphi}\left(p,\mathcal{C}\right)$, which in turn implies the desired result.
\end{proof}

The induced pressure can not only be formulated as in equation~\eqref{eq:induzierter-Druck_Definition}, but also in other equivalent ways, as stated in \cite{Jaerisch:2014}. One of these equivalent formulations is of particular interest, because it permits a simple proof of the sublinearity of the reciprocal escape rate $\rho^{-1} = 1/\rho$.
\begin{lem}[{{\cite[Corollary 2.2]{Jaerisch:2014}}}]
\label{lem:JKL_Corollary-2.2}Let $p,\varphi:S_{P}^{\mathbb{N}}\rightarrow\mathbb{R}$ be functions with $\varphi$ positive and bounded away from zero. The induced pressure with respect to  $\mathcal{C}\subset\Sigma_{P}^{*}$ can then be phrased as 
\begin{equation}
\mathcal{P}_{\varphi}\left(p,\mathcal{C}\right)=\inf\left\{ \beta\in\mathbb{R}\middle|\; \mathcal{P}_{\ind}\left(p-\beta\varphi,\mathcal{C}\right)\leq0\right\} .\label{eq:induzierter-Druck-inf-Formulierung}
\end{equation}
\end{lem}
In order to show the sublinearity of the reciprocal escape rate, one shows superadditivity of
\[
\varphi\mapsto\frac{1}{\mathcal{P}_{\varphi}\left(p,\mathcal{C}\right)}.
\]
\begin{prop}\label{prop:Superadditivit=0000E4t-induzierter-Druck}
Let $\varphi_{1},\varphi_{2}:X\rightarrow\mathbb{R}_{>0}$ be functions that are bounded and bounded away from zero, such that $-\infty<\mathcal{P}_{\varphi_{i}}\left(p,\mathcal{C}\right)<0$ holds for $i=1,2$. Then one obtains 
\[
\frac{1}{\mathcal{P}_{\varphi_{1}+\varphi_{2}}\left(p,\mathcal{C}\right)}\geq\frac{1}{\mathcal{P}_{\varphi_{1}}\left(p,\mathcal{C}\right)}+\frac{1}{\mathcal{P}_{\varphi_{2}}\left(p,\mathcal{C}\right)}.
\]
\end{prop}
\begin{proof}
Let $a_{i} \coloneqq\mathcal{P}_{\varphi_{i}}\left(p,\mathcal{C}\right)$ for $i=1,2$ and $x\coloneqq 1 / \left({a_{1}^{-1}+a_{2}^{-1}}\right)$.
With this notation one obtains 
\begin{align*}
 & \mathcal{P}_{\ind}\left(p-x\left(\varphi_{1}+\varphi_{2}\right),\mathcal{C}\right)\\
=\, & \mathcal{P}_{\ind}\left(p-x\varphi_{1}-x\varphi_{2},\mathcal{C}\right)\\
=\, & \mathcal{P}_{\ind}\left(p-\frac{x}{a_{1}}\cdot a_{1}\varphi_{1}-\frac{x}{a_{2}}\cdot a_{2}\varphi_{2},\mathcal{C}\right)\\
=\, & \mathcal{P}_{\ind}\left(p-\frac{x}{a_{1}}\cdot a_{1}\varphi_{1}-\left(1-\frac{x}{a_{1}}\right)\cdot a_{2}\varphi_{2},\mathcal{C}\right)\\
=\, & \mathcal{P}_{\ind}\left(\frac{x}{a_{1}}\cdot\left(p-a_{1}\varphi_{1}\right)+\left(1-\frac{x}{a_{1}}\right)\cdot\left(p-a_{2}\varphi_{2}\right),\mathcal{C}\right)\\
\overset{\left(\star\right)}{\leq}\, & \frac{x}{a_{1}}\cdot\mathcal{P}_{\ind}\left(p-a_{1}\varphi_{1},\mathcal{C}\right)+\left(1-\frac{x}{a_{1}}\right)\cdot\mathcal{P}_{\ind}\left(p-a_{2}\varphi_{2},\mathcal{C}\right)\\
=\, & 0+0,
\end{align*}
 with $\left(\star\right)$ following from the convexity of the induced pressure (\cite[Proposition 2.1]{Jaerisch:2014}). Thus Lemma~\ref{lem:JKL_Corollary-2.2} implies
\[
x\geq\mathcal{P}_{\varphi_{1}+\varphi_{2}}\left(p,\mathcal{C}\right)
\]
 and consequently 
\[
\frac{1}{\mathcal{P}_{\varphi_{1}}\left(p,\mathcal{C}\right)}+\frac{1}{\mathcal{P}_{\varphi_{2}}\left(p,\mathcal{C}\right)}\leq\frac{1}{\mathcal{P}_{\varphi_{1}+\varphi_{2}}\left(p,\mathcal{C}\right)}.
\]
\end{proof}
\begin{cor}\label{cor:AR-sublinear}
Under the conditions of Proposition~\ref{prop:AR-als-induzierter-Druck}, the reciprocal escape rate is sublinear.
\end{cor}
\begin{proof}
From Proposition~\ref{prop:EigenschaftenAR_Dach} (\ref{enu:EigenschaftenAR_Dach_Skalar}) it is known that the reciprocal escape rate is positive homogeneous. Propositions~\ref{prop:AR-als-induzierter-Druck} and \ref{prop:Superadditivit=0000E4t-induzierter-Druck} imply the subadditivity of the reciprocal escape rate. The combination of these two properties yields the sublinearity of the reciprocal escape rate.
\end{proof}

\section{Higher order asymptotics}\label{chap:Das-Umkehrproblem}

The local escape rate can be regarded as a first order asymptotic for the escape rate:
\[
\rho\left(B_{r}\left(x\right),\varphi\right) = \rho\left(x,\varphi\right) \cdot \mu\left(B_{r}\left(x\right)\right) + \oo{\mu\left(B_{r}\left(x\right)\right)}.
\]
Adapting a technique proposed by Cristadoro, Knight and Degli Esposti \cite{Cristadoro:2013} to the setting of special flows with locally constant ceiling functions over a Markov shift, one can compute higher order asymptotics. For periodic points, the second order term can be interpreted as the orbit length with respect to the flow.

\subsection{Calculating the escape rate}\label{sec:Cristadoro_Verfahren}

Let the base transformation be a Markov shift over a finite alphabet, let the ceiling function be $1$-arithmetic of the form 
\[ \varphi=\sum_{C\in C_{n}}k_{C}\cdot\ind_{C}\in Z_{n} \textrm{ with } k_C\in \mathbb{N}\]
and consider a hole $\bbar A\coloneqq A\times\left[0,1\right)$ with $A \in C_m$. Note that the vertical shape of the hole can be chosen freely because only its shadow is relevant for the escape rate.

The advantage of an arithmetic ceiling function and a hole of that particular shape is that the investigation of the escape rate can be simplified by looking at the discrete system $(\bbar X , \bbar \mu , \Phi_1)$ and its associated \emph{transfer operator} ${\bbar\L: L_{\bbar\mu}^{1}\left(\bbar{X}\right)\rightarrow L_{\bbar\mu}^{1}\left(\bbar{X}\right)}$  
 that is defined implicitly via $\int_{\bbar B}\bbar\L f\dx{\bbar\mu}=\int_{\Phi_{1}^{-1} \bbar B}f\dx{\bbar\mu}$ for all measurable sets $\bbar B\subset\bbar X$ and integrable functions $f\in L_{\bbar\mu}^{1}\left(\bbar{X}\right)$. The transfer operator of the open system is denoted by $\bbar{\L}_{\mathrm{op}}$ and defined as 
\[ \forall f\in L_{\bbar\mu}^{1}:\quad\bbar{\L}_{\mathrm{op}}f\coloneqq\bbar\L\left(\left(\ind-\ind_{\bbar A}\right)\cdot f\right). \]
Defining $\chi_n$ as the indicator function of $\left\{ (x,s)\in \bbar X\middle|\; \forall 0\leq k<n:\Phi_1^{k}\left(x,s\right)\notin \bbar A\right\} $, the measure of this set is given by the integral of $\chi_n$ and this can be phrased using the transfer operator of the open system: 
\[ \int\chi_{n}\dx{\bbar \mu} = \int\bbar{\L}_{\mathrm{op}}^{n}\ind\dx{\bbar \mu}. \]
The escape rate with respect to the hole $A$ is then given by 
\[ \rho\left(A,\varphi\right) = - \lim\limits_{n\rightarrow\infty} \frac{1}{n} \cdot \log \int\bbar{\L}_{\mathrm{op}}^{n}\ind\dx{\bbar \mu} \]
and thus intimately related to the spectral radius of $\bbar{\L}_{\mathrm{op}}$.

In this special setting, one can treat $\Phi_1$ as a Markov shift for the purpose of calculating the escape rate. The new alphabet is given by
\begin{equation}
\left\{ {C\times\left[k,k+1\right)} \, \middle|\; \, C\in C_{m},k\in\mathbb{N},0\leq k < k_{C}\right\},
\label{eq:New-Markov-Alphabet}
\end{equation}  
corresponding to the blocks of height $1$ that one obtains when using $n$-letter words to address blocks in the base. The new transition matrix is easily derived from that of the original Markov shift and provides a matrix representation of $\bbar \L$ with respect to the vector space spanned by the indicator functions of the sets listed in \eqref{eq:New-Markov-Alphabet}. If $\bbar A$ is one of these sets or a union thereof, one obtains a representation of $\bbar{\L}_{\mathrm{op}}$ from that of $\bbar \L$ by simply setting the entries in those rows that correspond to $\bar A$ to zero. Note that this method requires $m \leq n$. If $m > n$, one could move to a larger refined alphabet, but this would entail an exponential growth of the size of the matrix representation.

A better way of representing $\bbar{\L}_{\mathrm{op}}$ for small holes was devised by Cristadoro, Knight and Degli Esposti \cite{Cristadoro:2013}. The remainder of this subsection summarises the relevant results when adapted to the arithmetic special flow situation; a more detailed account can be found in \cite[Kapitel 7.1.1]{Dreher:2015}. Using their scheme, one only has to include enough information to represent the original closed system and $\bbar{\L}_{\mathrm{op}}^{k} \ind_{\bbar A}$ for ${k \geq 1}$. The latter does not add infinitely many new entries, because eventually these functions can be represented in terms of the original partition of the closed system. In fact $\bbar{\L}_{\mathrm{op}}^{k_0}$ can be written as a linear combination of the indicator functions of the sets in equation~\eqref{eq:New-Markov-Alphabet} if one chooses
\[
k_{0} \coloneqq \sum_{i=1}^{m-n}k_{\left[a_{i},\dots,a_{i+n-1}\right]} = S_{m-n}\varphi\left(x\right),
\]
the latter equality holding for $m>n$ and any $x \in A$ due to $\varphi$ being constant on $n$-cylinders.

In order to obtain a matrix representation of the action of $\bbar{\L}_{\mathrm{op}}$ on a vector space that contains $\ind$, it is natural to use the space spanned by
\begin{equation}
\left\{ \ind_{\bbar C}\middle|\; \bbar C\in\bbar C_{n}\right\} \cup\left\{\bbar{\L}^{k}\ind_{\bbar A}\middle|\; k=1,\dots,k_{0}-1\right\}. \label{eq:CristadoroBasiselementeEinfach}
\end{equation}
An additional condition on the hole $A = \left[a_1,\dots, a_m \right]$ is needed to ensure that the functions in \eqref{eq:CristadoroBasiselementeEinfach} are linearly independent: It is necessary that $\left(a_{1},\dots,a_{m}\right)\in\Sigma_{P}^{m}$ be reduced. This means that there is
$a_{m}'\neq a_{m}$ such that $\left(a_{1},\dots,a_{m-1},a_{m}'\right)\in\Sigma_{P}^{m}$. Using {\cite[Lemma 2 and 3]{Lind:1989}} one can then conclude that the elements of \eqref{eq:CristadoroBasiselementeEinfach} are linearly independent.

These observations are summarised in the following lemma.
\begin{lem}\label{lem:CristadoroBasis}
	Let the hole $A$ be determined by a reduced word $\left(a_{1},\dots,a_{m}\right)\in\Sigma_{P}^{m}$ with $m\geq n$ and let $k_{0} = S_{m-n}\varphi\left(x\right)$ for any $x\in A$.  
	Then 
	\begin{equation}
	\left\{ \frac{1}{\bbar{\mu}\left(\bbar C\right)}\cdot\ind_{\bbar C}\middle|\; \bbar C\in\bbar C_{n}\right\} \cup\left\{ \frac{1}{\bbar{\mu}\left(\bbar{\L}^{k}\ind_{\bbar A}\right)}\bbar{\L}^{k}\ind_{\bbar A}\middle|\; k=1,\dots,k_{0}-1\right\} \label{eq:CristadoroBasiselemente}
	\end{equation}
	 is the basis of a subspace $U$ of $L_{\bbar{\mu}}^{1}\left(\bbar X\right)$ that is mapped into itself by $\bbar{\L}_{\mathrm{op}}$ and that contains $\ind$.
\end{lem}

A necessary condition for $k_{0}>1$ is given by $m>n$, a sufficient one by $m>n+1$. If $m=n+1$, then it depends on the shape of the ceiling function. The scaling of the basis elements in \eqref{eq:CristadoroBasiselemente} is chosen such that the matrix representing $\bbar{\L}$ on $U$ has entries that can easily be derived from the transition matrix of the original system.  
Let $\bbar M$ be the matrix representing $\bbar{\L}$ on $\bbar Z_{n}$ with respect to the basis 
\[
\left\{ \frac{1}{\bbar{\mu}\left(\bbar C\right)}\cdot\ind_{\bbar C}\middle|\; \bbar C\in\bbar C_{n}\right\}. 
\]
The action of $\bbar{\L}_{\mathrm{op}}$ on $U$ with respect to the basis in Lemma~\ref{lem:CristadoroBasis} is given by the matrix 
\[\small{
\begin{blockarray}{c ccc cccc c}
		&	\BAmulticolumn{3}{c}{\bbar{C}_n} &	\bbar{\L}\ind_{\bbar{A}} & \BAmulticolumn{1}{r}{\cdots}& \BAmulticolumn{2}{r}{\bbar{\L}^{k_{0}-1}\ind_{\bbar{A}}} \\
	\begin{block}{c(ccc|cccc)c}
		\multirow{7}{*}{$\bbar{C}_n$} &	\BAmulticolumn{3}{c|}{\multirow{7}{*}{\Large$\bbar{M}$}} &	0 & \BAmulticolumn{3}{c}{\multirow{7}{*}{\Large$0$}} & \\
		&	&&&	\vdots	&&& &\\
		&	&&&	0	&&& &\\
		&	&&&	-\alpha	&&& & \substack{ \\ \leftarrow \left[a_1,\dots,a_n\right] \times \left[0,1 \right)}\\
		&	&&&	0	&&& &\\
		&	&&&	\vdots	&&& &\\
		&	&&&	0	&&& &\\
		\cline{2-8}
		\bbar{\L}\ind_{\bbar{A}}&	\BAmulticolumn{3}{c|}{\multirow{3}{*}{\Large$0$}} &	-c_1 & 1 && &\\
		\multirow{2}{*}{\vdots}&	&&&	\multirow{2}{*}{\vdots} &	&	\ddots & &\\
		&	&&&	&&& 1 &\\
		\bbar{\L}^{k_{0}-1}\ind_{\bbar{A}}&	0\,\cdots\,0&1&0\,\cdots\,0&	-c_{k_0 -1}&  0 & \cdots & 0 &\\
	\end{block}
		& \BAmulticolumn{3}{c}{\substack{\uparrow \\ \left[a_{\left( m-n+1\right)},\dots,a_m\right] \times \left[0,1 \right)}} &&&& & \\
\end{blockarray}}\]
which shall be referred to as $\bbar M_{\mathrm{op}}$. Applying $\bbar{\L}_{\mathrm{op}}$ corresponds to multiplying with $\bbar M_{\mathrm{op}}$ from the right. The values of the variables used in the matrix are as follows:
\begin{align*}
\alpha	& = \frac{\mu\left(\left[a_{1},\dots,a_{m}\right]\right)}{\mu\left(\left[a_{1},\dots,a_{n}\right]\right)}=p_{a_{n},a_{n+1}}\cdots p_{a_{m-1},a_{m}} \\
c_{k}	& = \frac{\bbar{\mu}\left(\bbar A\cap\Phi_{\lambda}^{k}\bbar A\right)}{\bbar{\mu}\left(\Phi_{\lambda}^{k}\bbar A\right)} = \begin{cases}
0 & ,\Phi_{\lambda}^{k}\bbar A\cap\bbar A=\emptyset\\
\frac{\mu\left(\left[a_{1},\dots,a_{l(k)+1}\right]\right)}{\mu\left(\left[a_{l(k)+1}\right]\right)}=p_{a_{1},a_{2}}\cdots p_{a_{l(k)},a_{l(k)+1}} & ,\Phi_{\lambda}^{k}\bbar A\cap\bbar A\neq\emptyset
\end{cases}
\end{align*}
with $l(k)$ being the number that fulfils 
\[
S_{l(k)}\varphi\left(x\right)=\sum_{j=0}^{l(k)-1}\varphi\circ\theta^{j}\left(x\right)= k
\]
for any (and thus all) $x\in A$.

Letting $\mathfrak{r}$ denote the spectral radius of a matrix, the escape rate with respect to the hole $A$ is
\[
\rho\left(A,\varphi\right)=-\log \mathfrak{r}\left(\bbar M_{\mathrm{op}}\right).
\]

The spectral radius corresponds to the modulus of the leading eigenvalue of $\bbar M_{\mathrm{op}}$. Since the operator could be written as a non-negative matrix operating on the larger space $Z_m$ without altering the leading eigenvalue, one knows by the Perron-Frobenius theorem that this leading eigenvalue is non-negative.  The spectral radius can be calculated as the zero of smallest modulus of 
\[
\zeta_{\mathrm{op}}^{-1}\left(z\right)\coloneqq\det\left(\mathrm{id}-z\cdot\bbar M_{\mathrm{op}}\right).
\]

Using the notation
\begin{align*}
\zeta_{\mathrm{cl}}^{-1} \left(z\right)	&	\coloneqq\det\left(\mathrm{id}-z\cdot\bbar M\right)  \\
k_{0} & \coloneqq S_{m-n}\varphi\left(x\right)\quad\textrm{for any }x\in A\nonumber \\
\phi_{A} \left(z\right)	&	\coloneqq 1+\sum_{k=1}^{k_{0}-1}c_{k}z^{k}\nonumber \\
C_{k,l}  \left(z\right)	 & \coloneqq\left(-1\right)^{k+l}\det\left[\mathrm{id}-z\cdot\bbar M\right]_{k,l}\nonumber \\
t & \coloneqq\mathrm{index\, of\, the\, row\, belonging\, to\,}\left[a_{1},\dots,a_{n}\right]\times\left[0,\lambda\right)\nonumber \\
r & \coloneqq\mathrm{index\, of\, the\, column\, belonging\, to\,}\left[a_{\left(m-n+1\right)},\dots,a_{m}\right]\times\left[0,\lambda\right)\nonumber \\
\alpha & \coloneqq p_{a_{n},a_{n+1}}\cdots p_{a_{m-1},a_{m}}\nonumber 
\end{align*}
 with $C_{k,l}$ denoting the cofactor of the matrix $\mathrm{id}-z\cdot\bbar M$ at $\left(k,l\right)$, one obtains
\begin{align}
\label{eq:Cristadoro_Zeta_op_abstrakt} \zeta_{\mathrm{op}}^{-1}\left(z\right) =\, & \det\left(\mathrm{id}-z\cdot \bbar{M}_{\mathrm{op}}\right)  \\
=\, & \zeta_{\mathrm{cl}}^{-1} \left(z\right) \cdot \phi_A\left(z\right) + C_{t,r}\cdot \alpha \cdot z^{k_0}. \nonumber
\end{align}

Note that $\zeta_{\mathrm{cl}}^{-1}\left(z\right)$ and $C_{k,l}\left(z\right)$ are independent of the choice of the hole $A$. $\zeta_{\mathrm{cl}}^{-1}$ depends on the hole via the indices $t,r$ of $C_{t,r}$ and the function $\phi_{A}$ which Cristadoro et al. call the weighted correlation polynomial.

\subsection{Shrinking holes}
Continuing in the setting of the previous subsection, one considers a sequence of shrinking holes. The corresponding ideas in \cite{Cristadoro:2013} are adapted to accommodate the shape of the underlying special flow.

Fix a periodic point $x = (x_1, x_2 , \dots) \in X$ of prime period $p$ with $(x_1, \dots, x_p)$ being a reduced word. The shrinking metric balls centred at $x$ correspond to cylinder sets $\left[x_{1}\right]$, $\left[x_{1},x_{2}\right]$, $\dots$ of increasing length. Considering holes  
\[
A_{\nu}\coloneqq\left[x_{1},\dots,x_{\nu p}\right],
\]
one can simplify the formula for $\phi_{A_{\nu}}$ because the periodicity of $x$ provides additional information about the $c_k$. In order to simplify calculations further one can assume without loss of generality that  the ceiling function is a cylinder function of order $n$ such that $n$ is a multiple of $p$. Set 
\begin{align}
m & \coloneqq\nu p,\;\;
o  \coloneqq S_{p}\varphi\left(x\right),\;\;
k_{0} \coloneqq S_{m-n}\varphi\left(x\right)=\left(\nu-\frac{n}{p}\right)\cdot o	\nonumber \\
s & \coloneqq\min\left\{ k\in\mathbb{N}\middle|\; ko\geq k_{0}\right\} =\nu-\frac{n}{p} \label{eq:Def_m_o_k0_periodischer-Fall} \\
\mu_{\nu} & \coloneqq\mu\left(A_{\nu}\right),\;\;
\mu_{w}  \coloneqq\mu\left(\left[x_{1},\dots,x_{p}\right]\right),\;\;
\mu_{t}  \coloneqq\mu\left(\left[x_{1},\dots,x_{n}\right]\right) \nonumber
\end{align}
and assume $\nu > n/p$ as to ensure $k_0 > 1$. This condition on $\nu$ is no restriction because one is interested in shrinking holes and thus $\nu \rightarrow \infty$.

It follows that 
\begin{equation}
c_{ko}=c_{o}^{k}=\left(p_{x_{1},x_{2}}\cdots p_{x_{p},x_{p+1}}\right)^{k}\quad\textrm{for }0\leq k\cdot o<k_{0}\label{eq:c_ko}
\end{equation}
and
\begin{equation}
c_{j}=0\quad\textrm{for }j\notin o\mathbb{N}_{0},\label{eq:c_j}
\end{equation}
because $p$ is the prime period of $x$ with respect to $\theta$ and thus $o$ is the prime period of $(x,0)$ with respect to $\Phi_1$. The condition $n\geq p$ ensures that there are no non-zero terms $c_{j}$ for $j>k_{0}-o$.  Using 
\[
\mu_{\nu}=\frac{\mu_w}{c_{o}}\cdot c_{o}^{\nu},
\]
one obtains 
\begin{equation}
c_{o}^{s}=\frac{\mu_{\nu}}{\mu_{w}}\cdot c_{o}^{1-n/p}.	\label{eq:c_o_s}
\end{equation}
This reduces the expression for $\phi_{A_{\nu}}$ to 
\begin{equation}
\phi_{A_{\nu}}\left(z\right) =1+\sum_{k=1}^{s-1}c_{ko}z^{ko}
  =\frac{1-\left(c_{o}z^{o}\right)^{s}}{1-c_{o}z^{o}} 
  =\frac{1-\frac{\mu_{\nu}}{\mu_{w}}\cdot c_{o}^{1-n/p}z^{os}}{1-c_{o}z^{o}}. \label{eq:Phi_A_nu_zusammengefasst} 
\end{equation}
 Since $\bbar M$ inherits the property of being an irreducible stochastic matrix from the transition matrix of the original Markov shift, { $\zeta_{\mathrm{cl}}^{-1} (z)$ }has a simple zero in $z=1$. Therefore it is possible to write 
\begin{equation}
\zeta_{\mathrm{cl}}^{-1}\left(z\right)=\left(1-z\right)\cdot G\left(z\right)\label{eq:zeta_cl_ist_(1-z)_mal_G}
\end{equation}
with a polynomial $G\left(z\right)$. This will be useful because looking at shrinking holes means looking at $z$ for values close to $1$. The notation $f_{\nu}$ is used in order to emphasize the dependency of the function on the hole $A_\nu$.

Using equations \eqref{eq:Cristadoro_Zeta_op_abstrakt}, \eqref{eq:c_o_s}, \eqref{eq:Phi_A_nu_zusammengefasst} and \eqref{eq:zeta_cl_ist_(1-z)_mal_G} as well as the equality  $\alpha=\mu_{\nu}/\mu_{t}$, one obtains 
\begin{align}
f_{\nu}\left(z\right) & = \zeta_{\mathrm{op}}^{-1}\left(z\right) \nonumber\\
 & =\zeta_{\mathrm{cl}}^{-1}\left(z\right)\cdot\phi_{A}\left(z\right)+C_{t,r}\left(z\right)\cdot\alpha\cdot z^{k_{0}}\nonumber \\
 & =\left(1-z\right)\cdot G\left(z\right)\cdot\frac{1-\frac{\mu_{\nu}}{\mu_{w}}\cdot c_{o}^{1-n/p}z^{os}}{1-c_{o}z^{o}}+C_{t,r}\left(z\right)\cdot\frac{\mu_{\nu}}{\mu_{t}}\cdot z^{k_{0}}\nonumber \\
 & =\underset{\eqqcolon g_{1}\left(z\right)}{\underbrace{\left(1-z\right)\cdot G\left(z\right)\cdot\frac{1}{1-c_{o}z^{o}}}} \label{eq:f_nu-als-g1_g2} \\
 & \quad\,+\mu_{\nu}\cdot\underset{\eqqcolon g_{2,\nu}\left(z\right)}{\left(\underbrace{\frac{1}{\mu_{t}}\cdot C_{t,r}\left(z\right)\cdot z^{k_{0}}-\frac{c_{o}^{1-n/p}}{\mu_{w}}\cdot\left(1-z\right)\cdot G\left(z\right)\cdot\frac{z^{os}}{1-c_{o}z^{o}}}\right)}\nonumber 
\end{align}
Note that $g_{1}$ does not depend on $\nu$ and while $g_{2,\nu}$ does depend on $\nu$ via $k_{0}$ and $s$, it does not contain a $\mu_{\nu}$ term.  
Fix a $k\in\mathbb{N}$. In order to express the zero of smallest modulus of $f_\nu$ in terms of  $\mu_{\nu}$, one aims to find expressions $s_{1},\dots,s_{k}$ (possibly depending on $\nu$) that satisfy the condition $s_{j}\left(\nu\right)\cdot\mu_{\nu}\rightarrow0$ for $\nu\rightarrow\infty$ and for which $f_{\nu}\bigl(1+s_{1}\mu_{\nu}+s_{2}\mu_{\nu}^{2}+\dots+s_{k}\mu_{\nu}^{k}\bigr)=\oo{\mu_{\nu}^{k}}$, implying that $1+s_{1}\mu_{\nu}+s_{2}\mu_{\nu}^{2}+\dots+s_{k}\mu_{\nu}^{k}$ is a good approximation of the desired zero. In order to determine the $s_k$, one writes $f_{\nu}(z)$ as a Taylor polynomial at $1$ with a remainder term, then inserts $z = 1+s_{1}\mu_{\nu}+s_{2}\mu_{\nu}^{2}+\dots+s_{k}\mu_{\nu}^{k}$, writes the resulting expression as a polynomial in $\mu_\nu$ and finally chooses $s_{1},\dots,s_{k}$ in such a way that the coefficients of $\mu_{\nu}^{l}$ for $0\leq l\leq k$ turn out to be zero.

\subsection{Higher order approximations for arithmetic cylinder functions}
The idea for obtaining an approximation of the zero describing the escape rate that was laid out at the end of the previous subsection will be carried out. In contrast to \cite{Cristadoro:2013}, the argument will be carried out rigorously and ultimately yield a little-$o$ statement regarding the quality of the approximation.

$f_{\nu}$ is a rational function in $z$ which does not have a pole in $z=1$. Consequently, one can express $f_\nu$ as a Taylor polynomial with remainder term: 
\[
f_{\nu}\left(z\right)=f_{\nu}\left(1\right)+\sum_{l=1}^{k}\frac{1}{l!}f_{\nu}^{\left(l\right)}\left(1\right)\cdot\left(z-1\right)^{l}+\frac{1}{\left(k+1\right)!}f_{\nu}^{\left(k+1\right)}\left(\xi\right)\cdot\left(z-1\right)^{k+1}
\]
 with a $\xi=\xi\left(z\right)\in\left(1,z\right)$. With 
\[
z=1+s_{1}\mu_{\nu}+s_{2}\mu_{\nu}^{2}+\dots+s_{k}\mu_{\nu}^{k}
\]
 one obtains via sorting by powers of $\mu_{\nu}$ that 
\begin{align}
f_{\nu}\left(z\right)= & \phantom{+\mu_{\nu}\,\,\cdot}\underset{=0}{\underbrace{g_{1}\left(1\right)}}\label{eq:f_nu_als_PR}\\
 & +\mu_{\nu}\cdot\left(s_{1}g_{1}'\left(1\right)+g_{2,\nu}\left(1\right)\right)\nonumber \\
 & +\mu_{\nu}^{2}\cdot\left(s_{2}g_{1}'\left(1\right)+\frac{1}{2}s_{1}^{2}g_{1}''\left(1\right)+g_{2,\nu}'\left(1\right)\cdot s_{1}\right)\nonumber \\
 & +\mu_{\nu}^{3}\cdot\ldots\nonumber \\
 & +\dots\nonumber 
\end{align}
Since $g_{1}'\left(1\right)=-G\left(1\right)/\left(1-c_{o}\right)\neq0$, it is possible to recursively determine the $s_{i}$ that make the coefficients of $\mu_{\nu},\dots,\mu_{\nu}^{k}$ disappear.

\begin{fact}
\label{fact:s_i-Allgemeine-Formel}With $g_{1}$ and $g_{2,\nu}$ as in equation~\eqref{eq:f_nu-als-g1_g2} and  $j_{1},\dots,j_{i}$ in $\mathbb{N}_{0}$, one obtains the following general formula for $s_{i}$, $1\leq i\leq k$:
\begin{align*}
s_{i}=\, & -\frac{1}{g_{1}'\left(1\right)}\cdot\left(\sum_{l=2}^{k}\sum_{\substack{j_{1}+j_{2}+\dots+j_{l}=l\\
j_{1}+2\cdot j_{2}+\dots+i\cdot j_{i}=i
}
}\frac{1}{j_{1}!\cdots j_{i}!}\cdot g_{1}^{\left(l\right)}\left(1\right)\cdot s_{1}^{j_{1}}\cdots s_{i}^{j_{i}}\right)\\
 & -\frac{1}{g_{1}'\left(1\right)}\cdot\left(\sum_{l=0}^{k}\sum_{\substack{j_{1}+j_{2}+\dots+j_{l}=l\\
j_{1}+2\cdot j_{2}+\dots+i\cdot j_{i}=i-1
}
}\frac{1}{j_{1}!\cdots j_{i}!}\cdot g_{2,\nu}^{\left(l\right)}\left(1\right)\cdot s_{1}^{j_{1}}\cdots s_{i}^{j_{i}}\right)
\end{align*}
\end{fact}
The terms $s_{1}$ and $s_{2}$ are of particular interest. The following relation between $C_{t,r}\left(1\right)$, $\mu_{t}$ and $G\left(1\right)$ helps simplify the resulting terms.

\begin{lem}
	It holds that $C_{t,r}\left(1\right)=\frac{1}{\mu\left(\varphi\right)}\cdot\mu_{t}\cdot G\left(1\right)$.
\end{lem}
The following proof differs from the argument suggested in \cite{Cristadoro:2013} in that it only requires comparatively simple matrix calculations.

\begin{proof}
	Let $N$ be the number of rows (and thus also columns) in $\bbar M$. Note that $\bbar M$ is an irreducible row-stochastic matrix.  Therefore it has a unique left eigenvector $\vect v = \left(v_1, \dots, v_N\right)\in\mathbb{R}_{>0}^{N}$ with $\sum_{i=0}^{N} v_i=1$. This eigenvector corresponds to the probability measure $\bbar{\mu}/\mu\left(\varphi\right)$ on $n$-cylinders. It is therefore sufficient to show that $C_{t,r}\left(1\right)=v_{t}\cdot G\left(1\right)$.
	
	First it is shown that for all $r=1,\dots,N$ 
	\[
	C_{t,r}\left(1\right)=C_{t,t}\left(1\right).
	\]
	To this end, one expands the determinant along the $t$-th row and plugs in $z=1$
	\begin{equation}
	0 = \det\left(\mathrm{id}-1\cdot \bbar{M} \right)=C_{t,t}\left(1\right)-\sum_{r=1}^{N}1\cdot m_{t,r}C_{t,r}\left(1\right). \label{eq:C_mu_G_Det0_Zeile}
	\end{equation}
	This evaluates to zero because $1$ is an eigenvalue of $\bbar{M}$. Since $\bbar{M}$ is row-stochastic, this can also be written as
	\begin{equation}
	0=\sum_{r=1}^{N}m_{t,r}\left(C_{t,t}\left(1\right)-C_{t,r}\left(1\right)\right).\label{eq:C_mu_G_Vektor}
	\end{equation}
	The validity of this equation relies only on the fact that $\bbar{M}$ is row-stochastic, it does not depend on the specific entries. Also, $C_{t,r}\left(z\right)$ does not depend on the entries in the $t$-th row of $\bbar{M}$. Therefore, one can replace the $t$-th row of $\bbar{M}$ by any other non-negative row $\bigl(m_{t,1}',\dots,m_{t,N}'\bigr)$ with row sum $1$. In particular one can choose a row that consists only of zeroes except for a single $s$ for which  $m_{t,r}'=1$. Modifying $\bbar{M}$ in such a way yields
	\begin{equation}
	C_{t,r}\left(1\right)=C_{t,t}\left(1\right)\label{eq:C_mu_G_KofaktorGleichheit_bei_1}
	\end{equation}
	for all $r=1,\dots,N$, completing the first step of the proof.
	
	This, together with the analogue of equation~\eqref{eq:C_mu_G_Det0_Zeile} for expanding along the $r$-th column instead of the $t$-th row leads to 
	\[
	\sum_{t=1}^{N}C_{t,t}\left(1\right)\cdot m_{t,r}=\sum_{t=1}^{N}C_{t,r}\left(1\right)\cdot m_{t,r}=C_{r,r}\left(1\right)
	\]
	for $s=1,\dots,N$ which implies that $\left(C_{t,t}\left(1\right)\right)_{t=1,\dots,N}$ is a left eigenvector for $\bbar{M}$ with respect to the eigenvalue $1$, and thus a scalar multiple of $\vect v$. Using the Jacobian formula for the derivative of the determinant (see for example \cite[Kapitel 8.3, Theorem 1]{Magnus:1991}), one obtains
	\[
	\frac{\mathrm{d}}{\mathrm{d}z}\left(\det\left(\mathrm{id}-z\cdot M\right)\right)=\sum_{t=1}^{N}\sum_{r=1}^{N}-m_{t,r}\cdot C_{t,r}\left(z\right).
	\] 
	From equation~\eqref{eq:zeta_cl_ist_(1-z)_mal_G} one obtains
	\[
	\frac{\mathrm{d}}{\mathrm{d}z}\left(\det\left(\mathrm{id}-z\cdot M\right)\right)=\frac{\mathrm{d}}{\mathrm{d}z}\left(\left(1-z\right)\cdot G\left(z\right)\right)=-G\left(z\right)+\left(1-z\right)\cdot G'\left(z\right).
	\]
	For $z=1$, one can combine this with equation~\eqref{eq:C_mu_G_KofaktorGleichheit_bei_1} and obtain
	\[
	G\left(1\right)=\sum_{t=1}^{N}\sum_{r=1}^{N}m_{t,r}\cdot C_{t,r}\left(1\right)=\sum_{t=1}^{N}C_{t,t}\left(1\right).
	\]
	This leads to
	\[
	G\left(1\right)\cdot\vect v=\left(C_{t,t}\left(1\right)\right)_{t=1,\dots,N} = \left(C_{t,r}\left(1\right)\right)_{t=1,\dots,N},
	\]
	implying the desired result.
\end{proof}

Using this lemma, one obtains the following explicit formulae for $s_{1}$ and $s_{2}$: 
\begin{align}
s_{1} & =-\frac{g_{2,\nu}\left(1\right)}{g_{1}'\left(1\right)}
   =\frac{1-c_{o}}{\mu\left(\varphi\right)} \nonumber  \\
s_{2} & =-\frac{\frac{1}{2}s_{1}^{2}g_{1}''\left(1\right)+g_{2,\nu}'\left(1\right)\cdot s_{1}}{g_{1}'\left(1\right)} \label{eq:s1_s2} \\
 & =\left(\frac{1-c_{o}}{\mu\left(\varphi\right)}\right)^{2}\cdot\left(k_{0}-\frac{G'\left(1\right)}{G\left(1\right)}-\frac{c_{o}o}{1-c_{o}}+\frac{C'_{t,r}\left(1\right)}{C_{t,r}\left(1\right)}+\frac{\mu\left(\varphi\right)\cdot c_{o}^{1-n/p}}{\mu_{w}\cdot\left(1-c_{o}\right)}\right)\nonumber 
\end{align}

Several lemmata will now lead up to a proof of Theorem~\ref{thm:Cristadoro_Nullstelle_kleinO}.

\begin{lem}
\label{lem:Abschaetzung_kleinste_NS_nach_oben}Assume that the local escape rate in $x$ exists and is finite. Let $z_{\nu}\in\mathbb{R}_{\geq0}$ be the zero of smallest modulus of $f_{\nu}$. Then there are  $\gamma\in\mathbb{R}$ and $\nu_{0}\in\mathbb{N}$ such that for all $\nu\geq\nu_{0}$ it holds true that
\[
1\leq z_{\nu}<1+\gamma\mu_{\nu}.
\]
\end{lem}
\begin{proof}
Using $z_{\nu}=e^{\rho\left(A_{\nu},\varphi\right)}$ and $0 \leq \rho\left(A_{\nu},\varphi\right) = \rho\left(x,\varphi\right) \cdot \mu_\nu + \oo{\mu_\nu}$, one obtains
\[
z_{\nu} = 1 + \rho\left(x,\varphi\right) \cdot \mu_\nu + \oo{\mu_\nu},
\]
which implies the existence of a suitable $\nu_0$ for any $\gamma > \rho\left(x,\varphi\right)$.
\end{proof}

The following uses the notation $g^{(l)}$ to denote the $l$-th derivative.
\begin{lem}
\label{lem:s-k_mal_mu_gegen_Null}Let $\epsilon\in\left(0,1\right)$
and $\gamma\geq0$. Then  
\begin{eqnarray*}
\lim_{\nu\rightarrow\infty}\sup_{1\leq z\leq1+\gamma\mu_{\nu}}\left|g_{1}^{\left(l\right)}\left(z\right)\right|\cdot\epsilon^{\nu} & = & 0,\\
\lim_{\nu\rightarrow\infty}\sup_{1\leq z\leq1+\gamma\mu_{\nu}}\left|g_{2,\nu}^{\left(l\right)}\left(z\right)\right|\cdot\epsilon^{\nu} & = & 0
\end{eqnarray*}
 for all $l\in\mathbb{N}_{0}$. For $1\leq j\leq k$ this implies
\[
\lim_{\nu\rightarrow\infty}s_{j}\left(\nu\right)\cdot\epsilon^{\nu}=0,
\]
and in particular
\[
\lim_{\nu\rightarrow\infty}s_{j}\left(\nu\right)\cdot\mu_{\nu}=0.
\]
\end{lem}
\begin{proof}
$g_{1}^{\left(l\right)}\left(z\right)$ is independent of $\nu$ for all $l\in\mathbb{N}_{0}$. Therefore the statement for $g_{1}^{\left(l\right)}$ follows from the fact that $g_{1}^{\left(l\right)}$ is continuous and thus bounded on a  sufficiently small neighbourhood of $1$. $g_{2}^{\left(l\right)}\left(z\right)$ depends on $\nu$ via $s$ and $k_{0}$. Each $g_{2,\nu}^{\left(l\right)}\left(z\right)$ is a linear combination of the functions $C_{t,r}\left(z\right)$, $\left(1-z\right)$, $G\left(z\right)$, $\left(1-c_{o}z^{o}\right)^{-1}$, $z^{k_{0}}$, $z^{os}$ and their derivatives. The first four of these (and their derivatives) do not depend on $\nu$ and thus are bounded on a sufficiently small neighbourhood of $1$ due to being continuous. Because the assumption $n\in p\mathbb{N}$ leads to $k_{0}=os$, it suffices to consider one of the remaining terms $z^{k_{0}}$ and $z^{os}$. Since $k_{0}$ tends to $\infty$ for $\nu\rightarrow\infty$, one can assume without loss of generality that $k_{0}>k$. This leads to  
\begin{eqnarray*}
\frac{\mathrm{d}^{l}}{\mathrm{d}z^{l}}\left(z^{k_{0}}\right) & = & z^{k_{0}-l}\cdot\prod_{j=0}^{l-1}\left(k_{0}-j\right),
\end{eqnarray*}
 and thus
\begin{align*}
\sup_{1\leq z\leq1+\gamma\mu_{\nu}}\left|\frac{\mathrm{d}^{l}}{\mathrm{d}z^{l}}\left(z^{k_{0}}\right)\left(z\right)\right| & \leq k_{0}^{l}\cdot\left(1+\gamma\mu_{\nu}\right)^{k_{0}}\\
 & \leq\left(\nu p\cdot\sup\varphi\right)^{l}\cdot\left(1+\gamma\mu_{w}c_{o}^{\nu-1}\right)^{\nu p\cdot\sup\varphi}.
\end{align*}
 The term $\left(\nu p\cdot\sup\varphi\right)^{l}$ is a polynomial expression in $\nu$ and $\left(1+\gamma\mu_{w}c_{o}^{\nu-1}\right)^{\nu p\cdot\sup\varphi}$ is a bounded function of $\nu$ on $\left[0,\infty\right)$ because
\begin{align*}
\left(1+\gamma\mu_{w}c_{o}^{\nu-1}\right)^{\nu p\cdot\sup\varphi} & =e^{\log\left(1+\gamma\mu_{w}c_{o}^{\nu-1}\right)\cdot\nu p\cdot\sup\varphi}\\
 & \leq e^{\gamma\mu_{w}c_{o}^{\nu-1}\cdot\nu p\cdot\sup\varphi}
\end{align*}
 and this tends to $1$ for $\nu\rightarrow\infty$. Hence, the statement for $g_{2,\nu}^{\left(l\right)}$ follows from the fact that $\epsilon^{\nu}\cdot P\left(\nu\right)\rightarrow0$ for any polynomial $P$ in $\nu$. The statement for $s_{j}$ is a simple consequence of $s_j$ being a sum of products of the previously considered terms.
 \end{proof}
 
\begin{lem}\label{lem:f'_von_Null_weg_Beschaenkt}
Let $\gamma>0$. Then there are $K>0$ and $\nu_{0}\in\mathbb{N}$ such that for all $\nu\geq\nu_{0}$:
\[
0<K\leq\inf_{1\leq z\leq1+\gamma\mu_{\nu}}\left|f_{\nu}'\left(z\right)\right|.
\]
\end{lem}
\begin{proof}
By definition: 
\[
f_{\nu}'\left(z\right)=g_{1}'\left(z\right)+\mu_{\nu}\cdot g_{2,\nu}'\left(z\right).
\]
 Since $\mu_{\nu}=\mu_{w}c_{o}^{\nu-1}$ and $c_{o}<1$, Lemma~\ref{lem:s-k_mal_mu_gegen_Null} implies
\[
\lim_{\nu\rightarrow\infty}\mu_{\nu}\cdot\sup_{1\leq z\leq1+\gamma\mu_{\nu}}\left|g_{2,\nu}'\left(z\right)\right|=0.
\]
 Hence it suffices to show that $\inf_{1\leq z\leq1+\gamma\mu_{\nu}}\left|g_{1}'\left(z\right)\right|>0$
for sufficiently large $\nu$. This is the case, because  
\begin{align*}
g_{1}'\left(z\right) & =\frac{\mathrm{d}}{\mathrm{d}z}\left(\left(1-z\right)\cdot G\left(z\right)\cdot\frac{1}{1-c_{o}z^{o}}\right)\left(z\right)\\
 & =-G\left(z\right)\cdot\frac{1}{1-c_{o}z^{o}}+\left(1-z\right)\cdot\underset{\eqqcolon\left(\star\right)}{\underbrace{\frac{\mathrm{d}}{\mathrm{d}z}\left(G\left(z\right)\cdot\frac{1}{1-c_{o}z^{o}}\right)}}.
\end{align*}
 The first summand is bounded away from zero for $\nu$ large enough (and thus $z$ close to $1$), since
$G\left(1\right)\neq0$. The second summand tends to zero for $\nu\rightarrow\infty$, because the term $\left(\star\right)$ is continuous and thus bounded on a sufficiently small neighbourhood of $1$, and  $\left|1-z\right|\leq\gamma\mu_{\nu}$ tends to zero. The existence of a constant $K$ that does not depend on $\nu$ is due to $g_{1}'$ not depending on $\nu$.
\end{proof}
 \begin{assumption}\label{as:periodisch-schrumpfendes-Loch}
	Let $x=(x_1,x_2,\dots)\in X$ be a periodic point of prime period $p$ such that $(x_1,\dots,x_p)$ is a reduced word and the local escape rate  $\rho\left(x,\ind\right)$ with respect to the base transformation exists. Assume that $A_{\nu}\coloneqq\left[x_{1},\dots,x_{\nu p}\right]$ for $\nu\in\mathbb{N}$ defines a descending sequence of holes whose measures are denoted by $\mu_{\nu}\coloneqq\mu\left(A_{\nu}\right)$.
\end{assumption}

\begin{thm}\label{thm:Cristadoro_Nullstelle_kleinO}
	Let Assumption~\ref{as:periodisch-schrumpfendes-Loch} be fulfilled, let the ceiling function $\varphi$ be $1$-arithmetic, let $k\in\mathbb{N}$, and let  $s_{i}\left(\nu\right)$ for $i=1,\dots,k$ be as in Fact~\ref{fact:s_i-Allgemeine-Formel}. Then for  $z_{\nu}\coloneqq e^{\rho\left(A_{\nu},\varphi\right)}$ it holds that 
\[
\left|z_{\nu}-\left(1+s_{1}\mu_{\nu}+s_{2}\mu_{\nu}^{2}+\dots+s_{k}\mu_{\nu}^{k}\right)\right|=\oo{\mu_{\nu}^{k}}.
\]
\end{thm}
\begin{proof}
First note that $s_{1}>0$ is constant and thus Lemma~\ref{lem:s-k_mal_mu_gegen_Null} implies that there is a $\gamma>0$ such that 
\[
1\leq1+s_{1}\mu_{\nu}+s_{2}\mu_{\nu}^{2}+\dots+s_{k}\mu_{\nu}^{k}\leq1+\gamma\mu_{\nu}
\]
 for sufficiently large $\nu$. $\gamma$ can be chosen large enough to be suitable for Lemma~\ref{lem:Abschaetzung_kleinste_NS_nach_oben}, hence fulfilling $1\leq z_{\nu}<1+\gamma\mu_{\nu}$ for sufficiently large $\nu$. Thus one can estimate the quality of the approximation for $z_\nu$ via
\begin{equation*}
\left|z_{\nu}-\left(1+s_{1}\mu_{\nu}+\dots+s_{k}\mu_{\nu}^{k}\right)\right| 
\leq\frac{\left|f_{\nu}\left(z_{\nu}\right)-f_{\nu}\left(1+s_{1}\mu_{\nu}+\dots+s_{k}\mu_{\nu}^{k}\right)\right|}{\inf_{1\leq z\leq1+\gamma\mu_{\nu}}\left|f_{\nu}'\left(z\right)\right|}.
\end{equation*}
 Since $f_{\nu}\left(z_{\nu}\right)=0$ and $f_{\nu}\left(1+s_{1}\mu_{\nu}+s_{2}\mu_{\nu}^{2}+\dots+s_{k}\mu_{\nu}^{k}\right)$ can be written as an expression of the form 
\[
\mu_{\nu}^{k+1}\cdot\left(\textrm{polynomial  in \ensuremath{\mu_{\nu}},  \ensuremath{s_{j}},  \ensuremath{g_{1}^{\left(l\right)}} and \ensuremath{g_{2,\nu}^{\left(l\right)}}}\right)
\]
according to the definition of the $s_{j}$ as the terms that make the coefficients of $\mu_{\nu},\dots,\mu_{\nu}^{k}$ in equation~\eqref{eq:f_nu_als_PR} disappear, one applies Lemmata~\ref{lem:s-k_mal_mu_gegen_Null} and \ref{lem:f'_von_Null_weg_Beschaenkt} in order to obtain $\bigl|z_{\nu}-\bigl(1+s_{1}\mu_{\nu}+s_{2}\mu_{\nu}^{2}+\dots+s_{k}\mu_{\nu}^{k}\bigr)\bigr|=\oo{\mu_{\nu}^{k}}$.
\end{proof}

\subsection{Second order results for general cylinder functions}\label{sec:Cristadoro_Folgerungen}

Despite its name, the local escape rate only carries a global information about the ceiling function, namely the integral $\mu \left(\varphi\right)$. One needs to look at the second order term so as to obtain more information. In general, one cannot expect to recover the value of the ceiling function in a point $x$, because the escape rate is invariant under addition of coboundaries (Proposition~\ref{prop:EigenschaftenAR_Dach} (\ref{enu:EigenschaftenAR_Dach_Coboundary})). Such a modification of the ceiling function leaves the orbit length of a point invariant though. And indeed it is possible to recover the orbit lengths of periodic points from the second order term of the escape rate.

\begin{thm}\label{thm:Explizite-erste-zwei-Ordnungen}
	Let Assumption~\ref{as:periodisch-schrumpfendes-Loch} be fulfilled and let the ceiling function $\varphi$ be a cylinder function. Then 
	\[
		\rho\left(A_{\nu},\varphi\right)
		=
		\frac{1-c_{o}}{\mu\left(\varphi\right)} \cdot \mu_{\nu}
		+
		\left(\frac{1-c_{o}}{\mu\left(\varphi\right)}\right)^2 \cdot S_{p}\varphi\left(x\right) \cdot \nu \cdot \mu_{\nu}^2 
		+ \oo{\nu \cdot \mu_{\nu}^2 }
	\]
	for $\nu \rightarrow \infty$ with $c_{o}=p_{x_{1},x_{2}} \cdots p_{x_{p},x_{p+1}}$ and $\mu_{\nu}=\mu\left(A_{\nu}\right)$.
\end{thm}

Note that Theorem~\ref{thm:Explizite-erste-zwei-Ordnungen} holds not only for cylinder functions, but due to Proposition~\ref{prop:EigenschaftenAR_Dach} (\ref{enu:EigenschaftenAR_Dach_Coboundary}) also for ceiling functions that differ from a positive cylinder function only by addition of a coboundary.

An immediate consequence of Theorem~\ref{thm:Explizite-erste-zwei-Ordnungen} is an explicit formula for the local escape rate of periodic points. For $\varphi = \ind$ this is a special case of the results in \cite{Keller:2009,Ferguson:2012} and reproduces the corresponding result in \cite{Cristadoro:2013}.

\begin{cor}
	\label{prop:LokaleAR_expliziter_Wert}  
		Let $x$ be a periodic point of prime period $p$ and let the ceiling function $\varphi$ be continuous. Then
	\[
	\rho\left(x,\varphi\right)=\frac{1-c_{o}}{\mu\left(\varphi\right)}
	\]
	with $c_{o}=p_{x_{1},x_{2}}\cdots p_{x_{p},x_{p+1}}$.
\end{cor}

\begin{proof}
	The result for $\varphi = \ind$ follows directly from Theorem~\ref{thm:Explizite-erste-zwei-Ordnungen} and can be extended to continuous ceiling functions via Theorem~\ref{thm:LokaleAR_IntegralMalBasis}.
\end{proof}

\begin{proof}[Proof of Theorem~\ref{thm:Explizite-erste-zwei-Ordnungen}]
The result will be first shown for arithmetic ceiling functions $\varphi$. Without loss of generality, one can assume that $\varphi$ is $1$-arithmetic. Using $\rho\left(A_{\nu},\varphi\right)=\log z_{\nu}$, Theorem~\ref{thm:Cristadoro_Nullstelle_kleinO} implies
\[
z_{\nu}= 1+s_{1}\cdot\mu_{\nu}+s_{2}\cdot\mu_{\nu}^{2} + \oo{\mu_{\nu}^{2}}
\]
and thus one obtains via the approximation $\log\left(1+x\right)=x-1/2\cdot x^{2}+\oo{x^{2}}$ for the logarithm and by plugging in the values for $s_1$ and $s_2$ from equation~\eqref{eq:s1_s2}, that
\begin{align}
\rho\left(A_{\nu},\varphi\right)
& 	= s_{1} \cdot\mu_{\nu}
	+ \left(s_{2}-\frac{1}{2}s_{1}^{2}\right) \cdot\mu_{\nu}^{2}
	+ \oo{\mu_{\nu}^{2}}	 \label{eq:zweite-Ordnung-mit-K}	\\
&	= \frac{1-c_{o}}{\mu\left(\varphi\right)} \cdot\mu_{\nu}
	+ \left(\frac{1-c_{o}}{\mu\left(\varphi\right)}\right)^2 \cdot S_{p}\varphi\left(x\right) \cdot\nu  \cdot\mu_{\nu}^2
	+ \oo{\nu \cdot \mu_{\nu}^{2}}.	\nonumber
\end{align}
In order to extend the result to general cylinder functions, one rephrases equation~\eqref{eq:zweite-Ordnung-mit-K} as
\begin{align*}
	\rho\left(x,\varphi\right)^2 \cdot S_{p}\varphi\left(x\right)
&	= \lim\limits_{\nu \rightarrow \infty}
	\frac	{\rho\left(A_{\nu},\varphi\right) - \rho\left(x,\varphi\right) \cdot \mu_{\nu}}
			{\nu \cdot \mu_{\nu}^{2}} \\
&	= \lim\limits_{\nu \rightarrow \infty}
	\frac	{\rho^{-1}\left(A_\nu,\varphi\right) \cdot \mu_{\nu} - \rho^{-1}\left(x,\varphi\right)}
			{\nu \cdot \mu_{\nu}}
	\cdot \left(-1\right) \cdot
	\frac{\rho\left(A_\nu,\varphi\right)}{\mu_{\nu}} \cdot \rho\left(x,\varphi\right) \\
&	= -\rho\left(x,\varphi\right)^2 \cdot \lim\limits_{\nu \rightarrow \infty}
	\frac	{\rho^{-1}\left(A_\nu,\varphi\right) \cdot \mu_{\nu} - \rho^{-1}\left(x,\varphi\right)}
			{\nu \cdot \mu_{\nu}}.
\end{align*}
Corollary~\ref{cor:AR-sublinear} states that $\rho^{-1}\left(A_\nu,\cdot\right)$ is a convex function for each $\nu$ and it follows from Theorem~\ref{thm:LokaleAR_IntegralMalBasis} that  $\rho^{-1}\left(x,\cdot\right)$ is linear. Hence,
\begin{equation}
	F_\nu\left(\varphi\right) = 
	\frac	{\rho^{-1}\left(A_\nu,\varphi\right) \cdot \mu_{\nu} - \rho^{-1}\left(x,\varphi\right)}
			{\nu \cdot \mu_{\nu}}
	\label{eq:konvexe-Hilfsfunktion}
\end{equation}
is a convex function for each $\nu$. Considering the open subset $C = \left\{\varphi \in Z_n \middle|\;  \varphi >0 \right\}$ in the finite dimensional vector space $Z_n$ of cylinder functions that are constant on $n$-cylinders, $\left(F_\nu\right)$ is a sequence of finite, convex functions such that on the dense subset $C' \subset C$ of arithmetic functions the pointwise limit exists and is finite. A result from convex analysis ({\cite[Theorem 10.8]{Rockafellar:1972}}) then implies that the pointwise limit exists for all $\varphi \in C$, that the limit function $F$ defined by the pointwise limits is finite and convex and that the convergence is uniform on each closed bounded subset of $C$. Since the $F_\nu$ are continuous on $C$, this implies that the limit function $F$ is continuous on $C$ as well. Thus equation~\eqref{eq:zweite-Ordnung-mit-K} holds for all $\varphi \in C$. Since this is true for all $n \in \mathbb{N}$, one obtains the desired result for general cylinder functions.
\end{proof}

The technique used in the proof of Theorem~\ref{thm:Explizite-erste-zwei-Ordnungen} to extend the result from arithmetic cylinder functions to general cylinder functions is limited to finite dimensional vector spaces and thus cannot be used to obtain Theorem~\ref{thm:Explizite-erste-zwei-Ordnungen} for general continuous ceiling functions.

In fact it is easy to see that the convergence of the sequence of functions given by  equation~\eqref{eq:konvexe-Hilfsfunktion} cannot be locally uniform with respect to the supremum norm. Consider an $\varepsilon$-neighbourhood of a ceiling function $\varphi$ with $0<\varepsilon<\inf\varphi$. For each $\nu$ one can find a function $\psi_{\nu}\in B_{\varepsilon}\left(\varphi\right)$ such that $\mu\left(\psi_{\nu}\right)=\mu\left(\varphi\right)$
and $\rho\left(A_{\nu},\psi_{\nu}\right)=\rho\left(A_{\nu},\varphi+\varepsilon/2\cdot\ind\right)$. This is a consequence of Proposition~\ref{prop:EigenschaftenAR_Mengen} which states that one can modify the function $\varphi+\varepsilon/2\cdot\ind$ on the union of finitely many preimages of $A_{\nu}$ (hence on a set of measure arbitrarily close to $1$) without changing the escape rate with respect to the hole $A_{\nu}$. The denominator in equation~\eqref{eq:konvexe-Hilfsfunktion} goes to zero whereas for such a sequence of $\psi_{\nu}$ the numerator converges towards the positive value
\[
\rho^{-1}\left(x,\varphi + \frac{\varepsilon}{2}\cdot\ind \right) - \rho^{-1}\left(x,\varphi\right)
=
\frac{\varepsilon}{2} \cdot \rho^{-1}\left(x,\ind\right).
\]
and thus the diagonal sequence given by $F_\nu \left(\psi_\nu\right)$ does not converge to a finite value, thus contradicting locally uniform convergence.

Note that one cannot conclude from this that Theorem~\ref{thm:Explizite-erste-zwei-Ordnungen} fails for general continuous functions. The lack of locally uniform convergence only implies that one cannot use a naive approximation argument to obtain such a result.

\end{document}